\documentclass[10pt,leqno]{article}

\usepackage{amsmath,amssymb,amsthm,mathrsfs,dsfont,mathtools}
\usepackage[margin=3cm]{geometry} %宽版时用，配合双倍行距

\usepackage{titlesec,hyperref}

\usepackage{color}

\usepackage{fancyhdr}
\titleformat{\subsection}{\it}{\thesubsection.\enspace}{1pt}{}

\newtheorem{theo}{Theorem}[section]
\newtheorem{lemm}[theo]{Lemma}
\newtheorem{defi}[theo]{Definition}

\newtheorem{prop}[theo]{Proposition}
\newtheorem{rema}[theo]{Remark}
\numberwithin{equation}{section}

\allowdisplaybreaks %允许公式分页显示

\begin{document}
\title{Global existence and optimal decay rate of weak solutions to the co-rotation Hooke dumbbell model
\hspace{-4mm}
}

\author{ Wenjie $\mbox{Deng}^1$ \footnote{E-mail: detective2028@qq.com},\quad
Zhaonan $\mbox{Luo}^1$\footnote{E-mail:   1411919168@qq.com} \quad and\quad
 Zhaoyang $\mbox{Yin}^{1,2}$\footnote{E-mail: mcsyzy@mail.sysu.edu.cn}\\
 $^1\mbox{Department}$ of Mathematics,
Sun Yat-sen University, Guangzhou 510275, China\\
$^2\mbox{Shenzhen}$ Campus of Sun Yat-sen University, Shenzhen 518107, China}

\date{}
\maketitle
\hrule

\begin{abstract}
In this paper, we mainly study global existence and optimal $L^2$ decay rate of weak solutions to the co-rotation Hooke dumbbell model. This micro-macro model is a coupling of the Navier-Stokes equation with a nonlinear Fokker-Planck equation. Based on the defect measure propagation method, we prove that the co-rotation Hooke dumbbell model admits a global weak solution provided the initial data under different integrable conditions. Moreover, we obtain optimal time decay rate in $L^2$ for the weak solutions obtained by the Fourier splitting method.  \\
%Our strategy relies on the results established in our previous work \cite{Zhang-Yin}.
\vspace*{5pt}

\vspace*{5pt}
\end{abstract}

\vspace*{10pt}

%\phantomsection
%\addcontentsline{toc}{section}{\contentsname}
%添加目录到书签
\tableofcontents

\section{Introduction}
In this paper, we consider the Hooke dumbbell model of polymeric fluids \cite{Bird1977,Doi1988,2017Global}:
\begin{align}\label{system}
\left\{
\begin{array}{ll}
\partial_t u + {\rm div}~(u\otimes u)-\nu\Delta u + \nabla P = {\rm div}~\tau(\psi),~~~{\rm div}~u=0, \\
\partial_t \psi + {\rm div}~(u\psi)-a \ {\rm div}_ q[\nabla_q (\frac{\psi}{\psi_{\infty}})\psi_{\infty}] =  {\rm div}_q[\sigma(u)\cdot q\psi] + \mu\Delta \psi,\\
\tau(\psi) = \int_{\mathbb{R}^d} q\otimes \nabla_q \mathcal{U} \psi dq - {\rm Id} ,~~~\mathcal{U} = \frac{1}{2}|q|^2,~~~\psi_{\infty} = e^{-\mathcal{U}},
\end{array}
\right.
\end{align}
where $u(t,x)$ stands for the velocity of the polymeric liquid and $\psi(t,x,q)$ denotes the distribution function for the internal configuration. Here $x\in \mathbb{R}^d$ or $\mathbb{T}^d$ and polymer elongation $q\in\mathbb{R}^d$, which means that the extensibility of the polymers is infinite. In addition, the parameters $a$, $\mu$ and $\nu$ are nonnegative constants. Denote rotation $\Omega(u)=\frac{\nabla u - (\nabla u)^{T}}{2}$ and deformation tensor $D(u)=\frac {\nabla u+(\nabla u)^T} {2}$. In general, taking $\sigma(u)=\nabla u$, while  $\sigma(u)=\Omega(u)$ for the co-rotation case. 

It is universally  known that the system \eqref{system} can be used to described the fluids coupling polymers. The system is of great interest in many branches of physics, chemistry, and biology, see \cite{Bird1977,Doi1988}. In this model, a polymer is idealized as an ``elastic dumbbell" consisting of two ``beads" joined by a spring that can be modeled by a vector $q$. The polymer particles are described by a probability distribution $\psi(t,x,q)$ satisfying that $\int_{\mathbb{R}^{d}}
\psi(t,x,q)dq =\int_{\mathbb{R}^d}\psi_0 dq$, which represents the distribution of particles' elongation vector $q\in \mathbb{R}^{d}$. Moreover, stress tensor $\tau$ is generated by the polymer particles effect. One can derive the following Oldroyd-B equation from system \eqref{system} with $\int_{\mathbb{R}^d}\psi_0 dq= 1$ :
\begin{align}\label{eqtau}
	\left\{\begin{array}{l}
		\partial_tu + u\cdot\nabla u+\nabla P = {\rm div}~\tau+\nu\Delta u,~~~~{\rm div}~u=0,~~~~~~~~~~~~~~~~\\[1ex]
		\partial_t\tau + u\cdot\nabla\tau+a\tau+Q(\nabla u,\tau)=b D(u)+\mu\Delta\tau,\\[1ex]
		u|_{t=0}=u_0,~~\tau|_{t=0}=\tau_0, \\[1ex]
	\end{array}\right.
\end{align}
with $Q(\nabla u, \tau)=\tau \Omega(u)-\Omega(u)\tau+b\left(D(u)\tau+\tau D(u)\right)$
and the co-rotation case means $b=0$. One can refer to  \cite{Bird1977,Doi1988,Masmoudi2008,Masmoudi2013} for more details.
\subsection{Reviews for the polymeric fluid models}

 Owing to its importance and challenging, the polymeric fluids have been extensitively investigated in recent decades. In the following paragraphs, we will review some impressive results from the mathematical analysis of the polymeric fluid models.

Take $\nu>0$ and $\mu=0$ in systems \eqref{system} and \eqref{eqtau}. The local well-posedness in Sobolev spaces with potential $\mathcal{U}(q)=(1-|q|^2)^{1-\sigma}$ for $\sigma>1$ was firstly investigated by M. Renardy \cite{Renardy}. Later, the local existence of a stochastic differential equation with potential $\mathcal{U}(q)=-k\log(1-|q|^{2})$ and $k>3$ for a Couette flow was proven by B. Jourdain et
al. \cite{Jourdain}. By virtue of defect measure propagation method, P. L. Lions and N. Masmoudi \cite{Lions-Masmoudi} showed the strong convergence of an approximating sequence and constructed global weak solutions for the Oldroyd-B model in the co-rotation case. Furthermore, global weak solutions to the FENE model and the FENE-P model were established by similar methods $(\text{see}~ \cite{Masmoudi2011,Masmoudi2013})$. By establishing new a priori estimation for 2D Navier-Stokes system
and a losing derivative estimate for the transport equation, J. Y. Chemin and N. Masmoudi \cite{lifespan} gave a sufficient condition of non-breakdown for an incompressible viscoelastic fluid of the Oldroyd type. We remark that these estimates are of great significance for proving the strong solutions of viscoelastic fluids. Under the co-rotational condition, the global well-posedness for 2D polymeric fluid models without any small assumptions on the initial data was obtained by N. Masmoudi et al. \cite{Masmoudi2dfluid}. In addition, Z. Lei et al. \cite{Z.Leinewmethod} advanced a new method to improve the blow-up criterion for viscoelastic systems of Oldroyd type given in \cite{lifespan}. We mention that this new method is much simpler and can be extensively used to solve other problems involving the prior losing derivative estimate.

Take $\nu=0$ and $\mu>0$ in systems \eqref{system} and \eqref{eqtau}. T. M. Elgindi and F. Rousset \cite{2015Elgindi} proved global regularity for the 2-D Oldroyd-B type model \eqref{eqtau}. Later on, T. M. Elgindi and J. Liu \cite{2015Elgindi1} obtained global strong solutions of \eqref{eqtau} under small initial data in Sobolev spaces when $d=3$. Regarding the 2-D co-rotational Oldroyd-B type model and its corresponding Hooke dumbbell model, the global existence with a class of large initial data was proven by W. Deng et al. \cite{DLY}.

The long time behavior for polymeric fluid models is of great concern by N. Masmoudi \cite{2016Masmoudi}. Take $\nu>0$ and $\mu=0$ in system \eqref{system}. The long time decay of the $L^2$ norm to the incompressible Hooke dumbbell models was studied by L. He and P. Zhang \cite{He2009}. They founded that the solutions tends to the equilibrium by $(1+t)^{-\frac{3}{4}}$ when the initial perturbation is additionally bounded in $L^1$. Recently, the $L^2$ decay of the velocity $u$ to the co-rotational FENE dumbbell model with potential $\mathcal{U}(R)=-k\log(1-(\frac{|R|}{|R_{0}|})^{2})$ was studied by M. Schonbek \cite{Hookeweak7}. She proved that velocity $u$ tends to zero in $L^2$ by  $(1+t)^{-\frac{d}{4}+\frac{1}{2}}$ $(d\geq 2)$ with the additional assumption that $u_0\in L^1$. Moreover, she conjectured
that the sharp decay rate should be $(1+t)^{-\frac{d}{4}}$. However, she failed to prove it because she could not use the bootstrap argument as that of \cite{Schonbek1985} caused by the additional stress tensor. More recently, this result was improved by W. Luo and Z. Yin in \cite{Luo-Yin,Luo-Yin2}, wherein they showed that the optimal long time decay rate of velocity $u$ in $L^2$ is $(1+t)^{-\frac{d}{4}}$. Regarding the 2-D Oldroyd-B type model \eqref{eqtau} with $\nu=0$ and $\mu>0$, W. Deng et al. \cite{DLY} showed the long time decay rate in $H^1$ for the global solutions constructed by T. M. Elgindi and F. Rousset in \cite{2015Elgindi}.

\subsection{Main results}
Let $\sigma(u)=\Omega$. One can verify that $(0,\psi_{\infty})$ with $\psi_{\infty}(q)=e^{-\frac{1}{2}|q|^2}$
is a trivial solution of system \eqref{system}. Take $\nu=a=1$ and $\mu=0$ in system \eqref{system}. Considering the perturbations near the global equilibrium
\begin{align*}
	~~~~~~~u=u~~ and ~~g=\frac {\psi-\psi_\infty} {\psi_\infty},
\end{align*}
we can rewrite system \eqref{system} as follows :
\begin{align}\label{eq0}
	\left\{
	\begin{array}{ll}
		\partial_t u + {\rm div}~(u\otimes u)-\Delta u + \nabla P = {\rm div}~\tau(g),~~~{\rm div}~u=0,\\
		\partial_t g + {\rm div}~(ug)-\mathcal{L} g= \frac{1}{\psi_{\infty}}{\rm div}_q[\Omega\cdot qg\psi_{\infty}],\\
		\tau(g) = \int_{\mathbb{R}^d} q\otimes \nabla_q \mathcal{U} g\psi_{\infty} dq,
	\end{array}
	\right.
\end{align}
where $\mathcal{L} g=\frac{1}{\psi_{\infty}}{\rm div}_ q[\nabla_q g\psi_{\infty}]$.
\begin{defi}\label{weakdefi}
Set $\varphi\in \mathscr{D}\left([0,T)\times\Lambda\right)$ and $\Phi\in \mathscr{D}\left([0,T)\times\Lambda\times\mathbb{R}^d\right)$ where $\Lambda = \mathbb{T}^d \ or \ \mathbb{R}^d$, then $(u,g)$ is said to be a weak solution for system \eqref{eq0} if the following conditions hold for any $T>0$ :\\
$(a)$ The velocity field u satisfies :
$$u\in C([0,T);L^2_w)\cap L^{\infty}([0,T);L^2)\cap L^2(0,T;\dot{H}^1).$$
For any $q\in[1,\infty)$ and $r=\frac{dq}{dq-2}$, the pressure $P$ satisfies : 
$$P\in L^1([0,T);W^{2,1})\cap L^2([0,T);W^{1,1})\cap L^q([0,T);L^r) + C([0,T);L^2).$$
Moreover, the stress tensor $\tau$ and the probability distribution $g$ satisfies :
 $$\tau\in C([0,T);L^2),~~~g\in C([0,T);L^2(\mathcal{L}^2)),~~~\nabla_qg\in L^2([0,T);L^2(\mathcal{L}^2)).$$ 
$(b)$ For any text function $\varphi$ and $\Phi$, it follows that
\begin{align*}
	\int_{\mathbb{R}}\int_{\Lambda} u\cdot(\partial_t\varphi +  u\cdot\nabla\varphi) - \nabla u\nabla\varphi + P~{\rm div}~\varphi dxdt+ \int_{\Lambda} u_0\varphi_0dx=  \int_{\mathbb{R}}\int_{\Lambda} \tau(g):\nabla \varphi dxdt,
\end{align*}
and
\begin{align*}
	\int_{\mathbb{R}}\int_{\Lambda}\int_{\mathbb{R}^d} \psi_{\infty}g(\partial_t\Phi + u\cdot\nabla\Phi)  - \nabla_q g\psi_{\infty}\cdot\nabla_q\Phi +\Omega\cdot q \psi_{\infty}g\cdot\nabla_q\Phi dqdxdt =-\int_{\Lambda} \int_{\mathbb{R}^d}\psi_{\infty}g_0\Phi_0dqdx.
\end{align*}
\end{defi}
~~

Our main results can be stated as follows.
\begin{theo}\label{th1}
Let $d = 2,3$ and $\Lambda = \mathbb{T}^d \ or \ \mathbb{R}^d$. Assume that a divergence-free field $u_0\in L^2$, $\int_{\mathbb{R}^d}g_0\psi_{\infty}dq=0$ and $\langle q\rangle g_0 \in (L^2\cap L^{\infty})(\mathcal{L}^2)$. Then system \eqref{eq0} admits a global weak solution $(u,g)$ in the sense of Definition \ref{weakdefi}. Moreover,
$u\in L^{\infty}(\mathbb{R}^{+};L^2)\cap L^{2}(\mathbb{R}^{+};\dot{H}^1)$, $\langle q\rangle g \in C\left([0,\infty);(L^2\cap L^{\infty})(\mathcal{L}^2)\right)$ and $\langle q\rangle\nabla_qg \in L^2\left(\mathbb{R}^{+};L^2(\mathcal{L}^2)\right)$.
\end{theo}

\begin{rema}
It follows from Lemma \ref{Lemma1} (see Section 2 below) that $\|\langle q\rangle g_0\|_{L^2(\mathcal{L}^2)} \leq C\| \nabla_qg_0\|_{L^2(\mathcal{L}^2)}$. Taking a divergence-free field $u_0\in L^2$ and $\nabla_qg_0 \in (L^2\cap L^{p})(\mathcal{L}^2)$ for some $p\geq4$, global existence of system \eqref{eq0} can be proven by the same way as the proof of Theorem \ref{th1}. However, the method shown in theorem above does not achieve the same result for the more critical cases $p\in (2,4)$, which needs to be solved by deriving a new renormalized equation introduced in Section 4. The main result is given in the following theorem.
\end{rema}

\begin{theo}\label{th2}
Let $d = 2,3$ and $\Lambda = \mathbb{T}^d$ or $\mathbb{R}^d$. Assume that a divergence-free field $u_0\in L^2$, $\int_{\mathbb{R}^d}g_0\psi_{\infty}dq=0$ and $\nabla_qg_0 \in (L^2\cap L^{p})(\mathcal{L}^2)$ for some $p>2$. Then system \eqref{eq0} admits a global weak solution $(u,g)$ in sense of Definition \ref{weakdefi}. Furthermore,
$u\in L^{\infty}(\mathbb{R}^{+};L^2)\cap L^{2}(\mathbb{R}^{+};\dot{H}^1)$, $\langle q\rangle g\in C([0,\infty);(L^2\cap L^{p})(\mathcal{L}^2))$ and $\nabla_qg \in C([0,\infty);(L^2\cap L^{p})(\mathcal{L}^2))$.
\end{theo}

\begin{rema}
	Take $g=0$, then the
	Hooke dumbbell model \eqref{eq0} is reduced to the Navier-Stokes equation. Theorems \ref{th1} and \ref{th2} cover the J. Leray celebrated results about global existence of weak solutions to the Navier-Stokes equation \cite{Leray, lions1996, lions1998}. When $d=2$, the high integrability of $g$ yield the results about uniqueness and further regularity, see \cite{Masmoudi2dfluid}.
\end{rema}

\begin{theo}\label{th3}
	Let $d = 2,3$ and $\Lambda = \mathbb{R}^d$. Let $(u,g)$ be global weak solution constructed in Theorem \ref{th1} or Theorem \ref{th2}. Suppose $u_0\in L^1$ and $g_0 \in L^1(\mathcal{L}^2)$, then there exist constants $c$ and $C$ such that
\begin{align}
~~\|u\|_{L^2}  \leq C(1+t)^{-\frac{d}{4}} \  \ and \ \ \ \|g\|_{L^2(\mathcal{L}^2)} \leq Ce^{-2ct}.
\end{align}
Moreover, if $\int_{\mathbb{R}^d}u_0dx \neq 0$ and $\|(u_0,\|g_0\|_{\mathcal{L}^2})\|_{L^2} \leq \delta$ for some small constant $\delta$, then there exists a constant $c$ such that
\begin{align}
	~~~~~~~\|u\|_{L^2}  \geq c(1+t)^{-\frac{d}{4}}.
\end{align}
\end{theo}

\begin{rema}
	In \cite{Schonbek1991}, M. Schonbek proved that $(1+t)^{-\frac{d}{4}}$ $(d=2,3)$ is the optimal decay rate in $L^2$ for the
	Navier-Stokes equations with the additional low frequency condition $u_0\in L^1$. Theorem \ref{th3} covers the M. Schonbek results of optimal decay rate in $L^2$ of weak solutions to the Navier-Stokes equation.
\end{rema}

\subsection{Motivation and main idea}
An open problem raised by N. Masmoudi in paper \cite{Masmoudi2013} is whether the Hooke dumbbell model admits a global weak solution to the general initial data in $L^2(\mathbb{R}^d)$ ?

For the general case where $\sigma(u) = \nabla(u)$, N. Masmoudi cannot use the same method as that of \cite{Masmoudi2013} to exclude the possibility of concentrated measures in nonlinear limits without the estimate of the $L^2(\mathbb{R}^d)$ norm to the stress tensor $\tau (g)$. There seems to be a long way to go from the conservation law of the known stress tensor $\tau(g)$ to its desired $L^2(\mathbb{R}^d)$-norm estimate. So far, there is still no effective method to address this issue.

For the co-rotation case where $\sigma(u) = \Omega$, due to the infinite elongation of the micro quantity $q$ in the main nonlinear term $\nabla_q\cdot(\sigma(u)\cdot q\psi_{\infty}g)$, it may lead to the generation of micro concentrated measures in the nonlinear limit. Secondly, unlike the handling of the dissipation term in paper \cite{Lions-Masmoudi}, the microscopic dissipation term $\mathcal{L}(g)$ may also lead to the generation of micro concentrated measures. In the renormalization equation, the sum of the concentrated measures generated by $\mathcal{L}(g)$ is not known to be positive or negative, which cannot be treated as an absorption term in the final estimation. To compensate for the lack of regularity at the microscopic level, we need to establish new a prior estimates using microscopic weights and construct new renormalization factors to obtain the required nonlinear limits under lower integrable initial conditions.\\

In this paper, we mainly focus on the following two issues :

1. Whether the Hooke dumbbell model admits a global weak solution to the general initial data under the co-rotation case ?

2. How to derive the decay rate of weak solutions obtained and further prove that the decay rate derived is optimal ?\\

Our main strategy for the two issues above are as follows :\\

Firstly, we focus on the analysis of the compactness of the velocity field $u$ :\\

As done in Section 2, we decompose velocity field $u = u_1 + u_2 + u_3$, where $u_1$, $u_2$, and $u_3$ satisfy the heat equations \eqref{eq1}, \eqref{eq2}, and \eqref{eq3} coupled with different nonlinear terms, respectively. Based on known energy conservation, we can derive an estimate of $u \cdot\nabla u$ in Lebesgue space. By using the standard heat kernel estimation, the estimates of $u_1$ and $u_2$ up to the second-order derivative in Lebesgue space can be obtained. From the compact embedding theorem, it is easy to obtain the strong convergence properties of $\nabla u_1$ and $\nabla u_2$ in Lebesgue space. However, for the velocity field $u_3$, the available prior estimates can only guarantee the boundedness of stress tensor $\tau(g)$ in the Lebesgue space. Therefore, the external force ${\rm div}~\tau(g)$ coupled by the heat equation \eqref{eq3} is only bounded in the first-order negative Lebesgue space merely. According to the standard heat kernel estimation, what we can obtain merely is the estimate of $u_3$ up to the first derivative in Lebesgue space, which is not sufficient to ensure the strong convergence property of $\nabla u_3$ in Lebesgue space. Anyway, ${\rm div}~\tau(g)$ does lead to the lack of compactness in the first-order derivative of the velocity field $\nabla u$.

Therefore, in order to obtain the weak convergence of the nonlinear term $\nabla_q\cdot(\sigma(u)\cdot q\psi_{\infty}g)$, what we can rely on is the strong convergence of probability distribution $g$. \\

Secondly, we focus on the analysis of the compactness of the probability distribution $g$ :\\

Consider defect measure $\eta$ such that $|g^n-g|^2 \rightarrow \eta$. We next aim to obtain $\eta = 0$. Our first step is to derive the possible defect measures that may arise in nonlinear convergence. There are the concentrated measures $\mu$ derived from $|\nabla (u^n - u)|^2$ and $\kappa$ derived from macroscopic dissipation term $|\nabla_q g^n - \nabla_q g|^2$, followed by the oscillated measures $\alpha$ derived from stress tensor $|\tau^n - \tau|^2$ and $\beta$ derived from mesoscopic coupling term $\langle q\rangle g^n \nabla u^n - \langle q\rangle g\nabla u$. Below, we provide different strategies for initial conditions under different integrability.\\

\textbf{For the case of initial data $\langle q\rangle g_0 \in (L^2 \cap L^\infty)(\mathcal{L}^2)$ :}\\

Referring to the significant mathematical discovery by P. L. Lions regarding the effective viscous pressure $\rm P_{eff}$ (see \cite{Feireisl2004,2Lions1998}), the following measure's identity is derived by renormalizing the equation $\eqref{eq0}_1$ : 
$$ \mu = -\int_{\mathbb{R}^d} \beta_{ji} \frac{q_i}{\langle q\rangle} \nabla_j\mathcal{U} \psi_{\infty} dq.
$$
This measure's identity can eliminate the concentration of measure $\mu$. That is to say, the measure $\mu$ derived from $| \nabla(u^n - u) |^2$ is an oscillated measure indeed. Based on the measure's analysis in detail by virtue of the measure's identity above, the following effective inequality of measures can be obtained :
$$
\left(\int_{\mathbb{R}^d} \langle q\rangle^2 |\beta|^2 \psi_{\infty}dq\right)^{\frac{1}{2}} \lesssim \int \eta \psi_{\infty} dq.
$$
According to the inequality obtained, it can be inferred that the oscillated measure $\eta$ is the largest among the oscillated measures involved. In other words, the oscillated measures $\alpha$, $\beta$ and $\mu$ can all be controlled by $\eta$. In order to observe the dynamic behavior of oscillated measure $\eta$, an attempt was made to identify the developmental equation that $\int_{\mathbb{R}^d} \eta \psi_{\infty} dq$ satisfies. By renormalizing the equation $\eqref{eq0}_2$, a sufficiently integrable renormalization equation satisfied by $\int_{\mathbb{R}^d} \eta \psi_{\infty} dq$ is derived:
$$
\partial_t \left(\int_{\mathbb{R}^d} \eta \psi_{\infty} dq\right) + {\rm div}~
\left(\int_{\mathbb{R}^d} \eta \psi_{\infty} dq\right) + 2\|\kappa\|_{\mathcal{M}(q)} = \int_{\mathbb{R}^d} (\beta_{ij}-\beta_{ji}) \frac{q_j}{\langle q\rangle} \nabla^i_q g \psi_{\infty} dq \in L^1_T(L^1_x).
$$
Moreover, Lemma \ref{Lions} admits a unique Diperna-Lions flow $X (t, x)$ such that $$
\frac{\partial}{\partial t}X(t,x)=u(t,X(t,x)).
$$ According to Mild formulation shown in Lemma \ref{Mild} and $\int_{\mathbb{R}^d} \eta \psi_{\infty}dq \in L^{\infty}((0, T)\times\Lambda)$ in hand, it can be inferred that 
$$
a.e.~~ x \in \Lambda,~~\int_{\mathbb{R}^d} \eta \psi_{\infty} dq (t, X(t, x)) \in BV(0, T).
$$
That is to say, we can observe the dynamic behavior of $\eta$ point by point under Diperna-Lions flow $X (t, x)$. It is continued to process the equation satisfied by $\int_{\mathbb{R}^d} \eta \psi_{\infty} dq$. Based on the inequality of measures mentioned above, we deduce that
$$
\int_{\mathbb{R}^d} (\beta_{ij}- \beta_{ji}) \frac{q_j}{\langle q\rangle} \nabla^i_q g \psi_{\infty} dq \lesssim \|\langle q\rangle\nabla_qg\|_{\mathcal{L}^2} \int_{\mathbb{R}^d} \eta \psi_{\infty} dq,~~a.e.~~x \in \Lambda.
$$
Combining the developmental equation satisfied by $\int_{\mathbb{R}^d} \eta \psi_{\infty} dq$, $\int_{\mathbb{R}^d} \eta \psi_{\infty} dq (t, X(t, x))$ satisfies the following inequality under the Diperna-Lions flow $X (t, x)$ :
$$
0 \leq \int_{\mathbb{R}^d} \eta \psi_{\infty} dq (t, X(t, x)) \lesssim \int_{\mathbb{R}^d} \eta_0 \psi_{\infty} dq \cdot e^{ \int_0^t \|\langle q\rangle\nabla_qg\|_{L^2} ds },~~a.e.~~x \in \Lambda,
$$
where we have thrown the concentrated measure $\|\kappa\|_{\mathcal{M}(q)}$ away during the estimation process above since its positive. According to $\int_{\mathbb{R}^d} \eta_0 \psi_{\infty} dq=0$ and the uniqueness of Diperna-Lions flow $X(t, x)$, the standard transport equation theory ensures $\int_{\mathbb{R}^d} \eta \psi_{\infty} dq=0$, or equivalently $\eta(t,x) = 0,~~(t,x)\in[0,T] \times\Lambda$.\\

\textbf{For the case of initial data $\nabla_q g_0 \in (L^2 \cap L^p)(\mathcal{L}^2)$, for some $p\in(2,\infty)$ :}\\

Unlike the first case mentioned above, $\int_{\mathbb{R}^d} \eta \psi_{\infty} dq \notin L^\infty ((0, T)\times\Lambda)$ under lower integrable initial data. It is difficult to carry out $\eta = 0$ by observing $\int_{\mathbb{R}^d} \eta \psi_{\infty} dq$ through the Mild formula under Diperna-Lions flow. Therefore, we need to reconstruct a new renormalization factor and derive its renormalization equation with sufficient integrability to adapt to the conditions required by the Mild formula. Let's consider the renormalization equation satisfied by the oscillated measure $N$ derived from $\|g^n\|^2_{\mathcal{L}^2} + 1$ :
$$
\partial_t N + {\rm div} (uN) + \frac{1}{N} \int_{\mathbb{R}^d}|\nabla_qg|^2 \psi_{\infty} dq + \int_{\mathbb{R}^d} \kappa \psi_{\infty} dq = 0,~~~N^2 = \|g\|^2_{\mathcal{L}^2} + \int_{\mathbb{R}^d} \eta \psi_{\infty} dq + 1.
$$
Although it seems impossible for $p\neq\infty$ to derive the boundedness of the $L^{\infty}((0, T)\times\Lambda)$-norm of $\int_{\mathbb{R}^d} \eta \psi_{\infty} dq$, a wonderful idea is to consider the renormalization factor $\frac{\int_{\mathbb{R}^d} \eta \psi_{\infty} dq}{N^2}$ instead, which naturally satisfies the boundedness in $L^{\infty} ((0, T)\times\Lambda)$. According to renormalization shown in Section 3 in detail, the following equation satisfied by $\frac{\int_{\mathbb{R}^d} \eta \psi_{\infty} dq}{N^2}$ is obtained :
\begin{align*}
\partial_t \left(\frac{\int_{\mathbb{R}^d} \eta \psi_{\infty} dq}{N^2}\right) &+ {\rm div} \left(u\frac{\int_{\mathbb{R}^d} \eta \psi_{\infty} dq}{N^2}\right) + 2 \mathop{\lim}\limits_{\delta \rightarrow 0^+} G_{\delta} \\ \notag 
&= \frac{1}{N^2}\int_{\mathbb{R}^d}(\beta_{ji} - \beta_{ij})\frac{q_j}{\langle q\rangle} \nabla^i_q g \psi_{\infty} dq + 2\frac{\int_{\mathbb{R}^d} \eta \psi_{\infty} dq}{N^4} \int_{\mathbb{R}^d} |\nabla_qg|^2 \psi_{\infty} dq \in L^1_T(L^1_x).
\end{align*}
In general, the measure sequence $G_\delta$ generated by the microscopic dissipation term $\mathcal{L}(g)$ cannot be controlled by $\eta$ and would converge to a certain concentrated measure. However, its positive or negative is unknown, which implies that the limit of $G_\delta$ cannot be treated as what the concentrated measure $\|\kappa\|_{\mathcal{M}(q)}$ does in the first case. Fortunately, through our newly established a prior estimate with microscopic weight $\nabla_q$, we have compensated for the lack of integrability of $G_\delta$ at the microscopic level, which allowing measure sequence $G_\delta$ to eventually converge to a non-negative integrable function. It is continued to process the equation satisfied by $\frac{\int_{\mathbb{R}^d} \eta \psi_{\infty} dq}{N^2}$. Based on the inequality of measures mentioned in the first case, we deduce that
\begin{align*}
\frac{1}{N^2}\int_{\mathbb{R}^d}(\beta_{ji} - \beta_{ij}) \frac{q_j}{\langle q\rangle} \nabla^i_q g \psi_{\infty} dq &+ 2\frac{\int_{\mathbb{R}^d} \eta \psi_{\infty} dq}{N^4} \int_{\mathbb{R}^d} |\nabla_qg|^2 \psi_{\infty} dq \\ \notag
&\lesssim \left(\|\langle q\rangle\nabla_qg\|_{\mathcal{L}^2} + \frac{\|\langle q\rangle\nabla_qg\|^2_{\mathcal{L}^2}}{N^2}\right)\frac{\int_{\mathbb{R}^d} \eta \psi_{\infty} dq}{N^2}~~ a.e.~~ x \in \Lambda.
\end{align*}
Combining the developmental equation satisfied by $\frac{\int_{\mathbb{R}^d} \eta \psi_{\infty} dq}{N^2}$, $\frac{\int_{\mathbb{R}^d} \eta \psi_{\infty} dq}{N^2}(t, X(t, x))$ satisfies the following inequality under the Diperna-Lions flow $X (t, x)$ determined by $u$ :
$$
0 \leq \frac{\int_{\mathbb{R}^d} \eta \psi_{\infty} dq}{N^2} (t, X(t, x)) \lesssim \frac{\int_{\mathbb{R}^d} \eta_0 \psi_{\infty} dq}{\|g_0\|^2_{\mathcal{L}^2}+\int_{\mathbb{R}^d} \eta_0 \psi_{\infty} dq+1} \cdot e^{ Ct + \int_0^t \|\langle q\rangle\nabla_qg\|_{\mathcal{L}^2}ds }~~ a.e.~~x\in\Lambda.
$$
According to $\int_{\mathbb{R}^d} \eta_0 \psi_{\infty} dq=0$ and the uniqueness of Diperna-Lions flow, the standard transport equation theory ensures $\frac{\int_{\mathbb{R}^d} \eta \psi_{\infty} dq}{N^2}(t,x)=0$, or equivalently $\eta(t,x) = 0$, $(t,x)\in[0,T]\times\Lambda$ by modifying the values of the sequence involved on a null set.

Regarding the global weak solutions obtained under different integrability mentioned above, their long-time behavior is also studied. Overall, we have proven the exponential decay rate for $\|g\|_{L^2(\mathcal{L}^2)}$ and the optimal decay rate for $\|u\|_{L^2}$. Firstly, we obtain initial logarithmic decay rate for $u$ in $L^2$ by additional energy estimates with micro weight and the Fourier splitting method. Then, by virtue of the logarithmic decay rate and the time weighted energy estimate, we improve the decay rate to $(1+t)^{-\frac{1}{2}}$. Finally, for certain initial data, the lower bound on the decay rate in $L^2$ for the corresponding velocity $u$ is established, which implies that the decay rate we obtained is optimal.
\subsection{Organization}
The rest of this paper is organized as follows. In Section 2, we introduce some notations and give some preliminaries which will be used in the sequel. In Section 3, we derive some priori estimates for system \eqref{eq0}. In Section 4, we present the compactness on velocity $u$ and probability distribution $g$. In Section 5, we present optimal decay rate in $L^2$ for global weak solutions of system \eqref{eq0} by virtue of the Fourier spiltting method.

%\vspace*{2em}
%\noindent\textbf{Acknowledgements}. This work was partially supported by ...

\section{Preliminaries}
 In this section, we will introduce some notations and useful lemmas which will be used in the sequel. We are only concerned with the case $\Lambda=\mathbb{R}^d$, since the periodic case is more easier. For $p\geq1$, we denote by $\mathcal{L}^{p}$ the space
$$\mathcal{L}^{p}\coloneqq\big\{f \big|\|f\|^{p}_{\mathcal{L}^{p}}=\int_{\mathbb{R}^d} \psi_{\infty}|f|^{p}dq<\infty\big\}.~~~~~~~~~~~~~~~~~~~~~$$
We will use the notation $L^{p}(\mathcal{L}^{q})$ to denote $L^{p}[\mathbb{R}^{d};\mathcal{L}^{q}]:$
$$L^{p}(\mathcal{L}^{q})\coloneqq\big\{f \big|\|f\|_{L^{p}(\mathcal{L}^{q})}=\left(\int_{\mathbb{R}^{d}}\left(\int_{\mathbb{R}^d} \psi_{\infty}|f|^{q}dq\right)^{\frac{p}{q}}dx\right)^{\frac{1}{p}}<\infty\big\}.$$
We now introduce the Littlewood-Paley decomposition theory and and Triebel-Lizorkin spaces.
\begin{prop}\cite{Bahouri2011,Compensated,Modern}\label{prop0'}
	Let $\mathcal{C}$ be the annulus $\{\xi\in\mathbb{R}^d:\frac 3 4\leq|\xi|\leq\frac 8 3\}$. There exist radial functions $\chi$ and $\varphi$, valued in the interval $[0,1]$, belonging respectively to $\mathscr{D}(B(0,\frac 4 3))$ and $\mathscr{D}(\mathcal{C})$, and such that
	$$ \forall\xi\in\mathbb{R}^d,\ \chi(\xi)+\sum_{j\geq 0}\varphi(2^{-j}\xi)=1, $$
	$$ \forall\xi\in\mathbb{R}^d\backslash\{0\},\ \sum_{j\in\mathbb{Z}}\varphi(2^{-j}\xi)=1,~~~ $$
	$$ |j-j'|\geq 2\Rightarrow\mathrm{Supp}\ \varphi(2^{-j}\cdot)\cap \mathrm{Supp}\ \varphi(2^{-j'}\cdot)=\emptyset, $$
	$$ ~~j\geq 1\Rightarrow\mathrm{Supp}\ \chi(\cdot)\cap \mathrm{Supp}\ \varphi(2^{-j}\cdot)=\emptyset. $$
	The set $\widetilde{\mathcal{C}}=B(0,\frac 2 3)+\mathcal{C}$ is an annulus, then
	$$ |j-j'|\geq 5\Rightarrow 2^{j}\mathcal{C}\cap 2^{j'}\widetilde{\mathcal{C}}=\emptyset. $$
	Further, we have
	$$ ~~\forall\xi\in\mathbb{R}^d,\ \frac 1 2\leq\chi^2(\xi)+\sum_{j\geq 0}\varphi^2(2^{-j}\xi)\leq 1, $$
	$$ \forall\xi\in\mathbb{R}^d\backslash\{0\},\ \frac 1 2\leq\sum_{j\in\mathbb{Z}}\varphi^2(2^{-j}\xi)\leq 1.~~ $$
$\mathscr{F}$ represents the Fourier transform and  its inverse is denoted by $\mathscr{F}^{-1}$.
Let $u$ be a tempered distribution in $\mathcal{S}'(\mathbb{R}^d)$. For all $j\in\mathbb{Z}$, define
$$
\dot{\Delta}_j u=0\,\ \text{if}\,\ j\leq -2,\quad
\Delta_{-1} u=\mathscr{F}^{-1}(\chi\mathscr{F}u),\quad
\dot{\Delta}_j u=\mathscr{F}^{-1}(\varphi(2^{-j}\cdot)\mathscr{F}u)\,\ \text{if}\,\ j\geq 0,\quad
\dot{S}_j u=\sum_{j'<j}\dot{\Delta}_{j'}u.
$$
Then the $\rm Littlewood$-$\rm Paley$ decomposition is given as follows:
$$ u=\sum_{j\in\mathbb{Z}}\dot{\Delta}_j u \quad \text{in}\ \mathcal{S}'(\mathbb{R}^d). $$
Let $s\in\mathbb{R},\ 1\leq p,r\leq\infty.$ The nonhomogeneous $\rm Triebel$-$\rm Lizorkin$ Space $F^s_{p,r}$ is defined by
$$ ~\dot{F}^s_{p,r}\coloneqq\{u\in S':\|u\|_{F^s_{p,r}}=\Big\|\|(2^{js}\dot{\Delta}_j u)_j \|_{l^r(\mathbb{Z})}\Big\|_{L^p}<\infty\}.~~~~~ $$
In particular,
$$ \mathcal{H}^1\coloneqq \dot{F}^0_{1,2}=\{u\in S':\|u\|_{\mathcal{H}^1}=\Big\|\|(\dot{\Delta}_j u)_j \|_{l^2(\mathbb{Z})}\Big\|_{L^1}<\infty\} $$
and
$$ ~~{\rm BMO}\coloneqq \dot{F}^0_{\infty,2}=\{u\in S':\|u\|_{{\rm BMO}}=\Big\|\|(\dot{\Delta}_j u)_j \|_{l^2(\mathbb{Z})}\Big\|_{L^{\infty}}<\infty\},$$
with duality ${\mathcal{H}^1}^{*}={\rm BMO}$. The following embedding hold
$$
\mathcal{H}^1 \hookrightarrow L^1 \ \ and \ \ \ L^{\infty} \hookrightarrow {\rm BMO}.
$$
If $\{u_i\}^2_{i=1}$ satisfies ${\rm div}~u_1=0$ and $\nabla\times u_2 = 0$, then there exists a positive constant $C$ such that
$$
\|u_1\cdot u_2\|_{\mathcal{H}^1} \leq C\|u_1\|_{L^2}\|u_2\|_{L^2}.~
$$
Moreover, denote $R_k$ be $\rm Riesz$ operator such that $\widehat{R}_k = \frac{\xi_k}{|\xi|}$. Then for $n \geq 2$, there exists a positive constant $C_n$ such that for $f\in \mathcal{H}^1(\mathbb{R}^n)$, we have
\begin{align}
 C_n^{-1}\|R_kf\|_{\mathcal{H}^1} \leq   \|f\|_{\mathcal{H}^1} \leq C_n (\|f\|_{L^1} + \displaystyle\sum_{m=1}^n \|R_mf\|_{L^1}).~
\end{align}
Note that the homogeneous $\rm Triebel$-$\rm Lizorkin$ Space is defined by $\dot{F}^s_{p,r}$ and $\dot{H}^s=\dot{F}^s_{2,2}$.
\end{prop}

We now introduce some notations about $\rm Lorentz$ spaces.
\begin{prop}\cite{Bahouri2011}\label{prop0}
	Let $f$ be a measurable function on a measure space $(X,\mu)$ and $0 < p,q\leq\infty$. The distribution function of $f$ is defined  as $d_f(\alpha) = \mu\left(\{x\in X:|f(x)|>\alpha\}\right).$
	The decreasing
	rearrangement of $f$ is defined as $f^{*}=\inf\{s>0:d_f(s)\leq t\}.$
	Set
	\begin{align}
	\|f\|_{L^{p,q}} \coloneqq \left\{
	\begin{array}{ll}
		\big(\int_0^{\infty}\left(t^{\frac{1}{p}}f^{*}(t)\right)^q\frac{dt}{t}\big)^{\frac{1}{q}} \ \  {\rm if} \ \ q<\infty,\\
		\mathop{\sup}\limits_{t>0} t^{\frac{1}{p}}f^{*}(t) \  \ \ \ {\rm if} \ \ q=\infty,
	\end{array}
	\right.
	\end{align}
    and $\rm Lorentz$ space $L_{p,q}(X,\mu)=\{f: \|f\|_{L^{p,q}}<\infty\}$. Note that $L^{p,p}=L^{p}$. Suppose $0 < q < r \leq \infty$, then there exists a constant $C_{p,q,r}$ such that
    \begin{align}
    ~\|f\|_{L^{p,r}} \leq C_{p,q,r}\|f\|_{L^{p,q}}.
    \end{align}
\end{prop}

The interpolation lemma is as follows.
\begin{lemm}\cite{1959On}\label{Lemma3}
	Denote that $\Lambda^{s}f=\mathscr{F}^{-1}\left(|\xi|^s\mathscr{F}f\right)$. Let $d\geq2,~p\in[2,+\infty)$ and $0\leq s,s_1\leq s_2$, then there exists a constant $C$ such that
	$$~~~~~~~~~~\|\Lambda^{s}f\|_{L^{p}}\leq C \|\Lambda^{s_1}f\|^{1-\theta}_{L^{2}}\|\Lambda^{s_2} f\|^{\theta}_{L^{2}},$$
	where $0\leq\theta\leq1$ and $\theta$ satisfy
	$$ s+d(\frac 1 2 -\frac 1 p)=s_1 (1-\theta)+\theta s_2.~~~$$
	Note that we require that $0<\theta<1$ and $0\leq s_1\leq s$ when $p=\infty$.
\end{lemm}
The following lemma allows us to estimate distribution $g$.
\begin{lemm}\cite{2017Global}\label{Lemma1}
Let $g\in L^2(\mathcal{L}^{2})$ with $\int_{\mathbb{R}^{d}} g\psi_\infty dq=0$. There exists a positive constant $C$ such that
\begin{align}
\|\langle q\rangle g\|_{L^2(\mathcal{L}^{2})}\leq C\|\nabla_qg\|_{L^2(\mathcal{L}^{2})},~~~~\||q|^2g\|_{{L^2(\mathcal{L}^{2})}}\leq C\|\langle q\rangle\nabla_qg\|_{L^2(\mathcal{L}^{2})},~~~~~~~~~
\end{align}
\end{lemm}

The following lemma is useful for showing optimal decay rate.
\begin{lemm}\cite{decayestimate}\label{Ldecay}
	Let $r_1 > 1,r_2\in[0,r_1]$. Then there exists a positive constant $C$ such that
	\begin{align*}
		\left\{\begin{array}{l}
			\int_0^t(1+s)^{-r_2}e^{-(1+t-s)} ds \leq C(r_2)(1+t)^{-r_2},\\
			\int_0^t (1+t-s)^{-r_1}(1+s)^{-r_2}ds \leq C(r_1,r_2)(1+t)^{-r_2}.
		\end{array}\right.
	\end{align*}
\end{lemm}
We give a commutator lemma, which is useful in renormalization process as that of \cite{lions1996}.
\begin{lemm}\label{renormalized1}
	Let $\Omega \subset \mathbb{R}^d$ be an arbitrary domain. Assume
	\begin{align}
		~f \in \mathcal{L}^2 \ \ \ and \ \ \ \psi^{\pm1}_{\infty} \in W^{1,2}(\Omega;\mathbb{R}^d)
	\end{align}
	be given functions with $\psi^{\pm1}_{\infty}$ bounded in any $K \subset\subset \Omega$. Then
	\begin{align}\label{bound}
		\|[\nabla_q\left(f\psi^{\pm1}_{\infty}\right)]^{\varepsilon}_q - \nabla_q\left([f]^{\varepsilon}_q\psi^{\pm1}_{\infty}\right)\|_{L^1(K)} \leq C(K)\|f\|_{\mathcal{L}^2}\|\psi^{\pm1}_{\infty}\|_{W^{1,2}(\Omega;\mathbb{R}^d)} ,
	\end{align}
	and
	\begin{align}\label{limit}
		[\nabla_q\left(f\psi^{\pm1}_{\infty}\right)]^{\varepsilon}_q - \nabla_q\left([f]^{\varepsilon}_q\psi^{\pm1}_{\infty}\right) \rightarrow 0 \ in \ L^1(K) \ \ as \ \ \varepsilon\rightarrow 0.~~~~~~~~
	\end{align}
Here $v \mapsto [v]^{\varepsilon}_q = \theta^{\varepsilon} \ast_q v$  is smoothing operator with $$~\theta^{\varepsilon}(x) = \varepsilon^{-d}\theta(\frac{x}{\varepsilon}) \ \ and \ \ \  \theta=\frac{e^{\frac{1}{|x|^2-1}}}{\int_{\mathbb{R}^d}e^{\frac{1}{|x|^2-1}}dx}~1_{\{|x|\leq 1\}}.~~~
$$
\end{lemm}
\begin{proof}
	To begin with, we observe that the following quantity
	\begin{align}
		[\nabla_q\left(f\psi^{\pm1}_{\infty}\right)]^{\varepsilon}_q - \nabla_q\left([f]^{\varepsilon}_q\psi^{\pm1}_{\infty}\right)~~~~~~~~~~~~~~~~~
	\end{align}
	is well defined on $K$ whenever $\varepsilon$ is sufficiently small. Moreover, \eqref{limit} holds for any $f\in C^{\infty}$, which is dense in $L^1(K)$. According to Banach-Steinhaus theorem, we ensure \eqref{limit} by showing the bound \eqref{bound}. To this end, we write
	\begin{align}
		[\nabla_q\left(f\psi^{\pm1}_{\infty}\right)]^{\varepsilon}_q &- \nabla_q\left([f]^{\varepsilon}_q\psi^{\pm1}_{\infty}\right) =-[f]^{\varepsilon}_q\nabla_q\psi^{\pm1}_{\infty}  \\ \notag
		& + \int_{\mathbb{R}^d}f(q-z)\frac{\psi^{\pm1}_{\infty}(q)-\psi^{\pm1}_{\infty}(q-z)}{|z|}\cdot\nabla_q\theta^{\varepsilon}(|z|)|z|dz.~~~~~~~~~~~
	\end{align}
	According to Minkowski's inequality, we infer that
	\begin{align}
		\int_{K}\int_{\mathbb{R}^d}f(q-z)\frac{\psi^{\pm1}_{\infty}(q)-\psi^{\pm1}_{\infty}(q-z)}{|z|}\cdot\nabla_q\theta^{\varepsilon}(|z|)|z|dzdx \leq C(K)\|f\|_{\mathcal{L}^2}\|\psi^{\pm1}_{\infty}\|_{W^{1,2}(\Omega;\mathbb{R}^d)}.
	\end{align}
	This completes the proof of Lemma \ref{renormalized1} .
\end{proof}

The following proposition is a straightforward consequence of Lemma \ref{renormalized1} .
\begin{prop}\label{renormalized2}
	Let $\Omega \subset \mathbb{R}^d$ with $d\geq 2$ be an arbitrary domain. Let
	\begin{align*}
	g \in C\left([0,T);L^2(\Omega,\mathcal{L}^2)\right) \ \ &and \  \ \nabla_qg \in  L^2\left(0,T;L^2(\Omega,\mathcal{L}^2)\right), \\
	u \in L^2\left(0,T;W^{1,2}(\Omega)\right) \ \ &and \ \ h \in L^1\left(0,T;L^1(\Omega,\mathcal{L}^2)\right),
   \end{align*}
	satisfy
	\begin{align}\label{renormalizedeq1}
	~~~~~\partial_t	g+ {\rm div}_x(g u) - \mathcal{L} g= h \ \ in \ \ \mathscr{D}^{'}\left((0,T)\times\Omega\times\Omega\right),
	\end{align}
	with $\mathcal{L} g = \frac{1}{\psi_{\infty}} \nabla_q\cdot\left(\nabla_qg\psi_{\infty}\right)$. Then we have
	\begin{align}\label{renormalizedeq2}
		\partial_t B(g) + {\rm div}_x\left(B(g)u\right) - \mathcal{L}B(g) + B^{''}(g)|\nabla_q g|^2= B'(g)h  \ \ in \ \ \ \mathscr{D}^{'}\left((0,T)\times\Omega\times\Omega\right),~~
	\end{align}
	where we take $B \in C^2[0,\infty)$ such that \eqref{renormalizedeq2} makes sense.
\end{prop}
\begin{proof}
	Firstly, we prove \eqref{renormalizedeq2} for any $B\in \mathscr{D}(0,\infty)$. Applying the regularizing operators $v \mapsto [v]_{x,q}^{\varepsilon}$ to both sides of \eqref{renormalizedeq1}, we obtain
	\begin{align}\label{g}
		\partial_t [g]_{x,q}^{\varepsilon} + \nabla_x [g]_{x,q}^{\varepsilon} \cdot u - \mathcal{L} [g]_{x,q}^{\varepsilon} = [h]_{x,q}^{\varepsilon} + s^{\varepsilon}+r^{\varepsilon} \ \ a.e. \ \ on \ \ O,~~~~~~~~~
	\end{align}
	for any bounded open set $O \subset \bar{O} \subset (0,T)\times\mathbb{R}^d\times\mathbb{R}^d$ provided $\varepsilon >0$ small enough with
	\begin{align}
		&s^{\varepsilon} = [\mathcal{L} g]_{x,q}^{\varepsilon}-\mathcal{L} [g]_{x,q}^{\varepsilon} \in L^1(O),	\\ \notag
		&r^{\varepsilon} ={\rm div}_x~\big\{[(gu)]^{\varepsilon}_{x,q} - ([g]^{\varepsilon}_{x,q}u)\big\} \in L^1(O).
	\end{align}
It follows from the proof of Lemma \ref{renormalized1} that 
\begin{align}
	[f\frac{q}{\psi_{\infty}}]^{\varepsilon}_q -[f]^{\varepsilon}_q\frac{q}{\psi_{\infty}} \rightarrow 0 \ \ in \ \ L^1(O) \ \  as  \ \ \ \varepsilon\rightarrow 0.~~~~~
\end{align}
Set $F=\nabla_qg\psi_{\infty}\in \mathcal{L}^2$. Notice that
\begin{align}
s^{\varepsilon} = [F\frac{q}{\psi_{\infty}}]^{\varepsilon}_q - [F]^{\varepsilon}_q\frac{q}{\psi_{\infty}} 
	+{\rm div}_q\left([F\psi^{-1}_{\infty}]^{\varepsilon}_q - [F]^{\varepsilon}_q\psi^{-1}_{\infty}\right)+\psi^{-1}_{\infty}{\rm div}_q\left([\nabla_qg\psi_{\infty}]^{\varepsilon}_q-[\nabla_qg]^{\varepsilon}_q\psi_{\infty}\right).
\end{align}
According to Lemma \ref{renormalized1}, we thus deduce that
	$$
	s^{\varepsilon} + r^{\varepsilon} \rightarrow 0  \ \ in \ \ L^1(O) \ \ as \ \ \varepsilon  \rightarrow 0.~~~
	$$
	Multipling $B^{'}([g]_{x,q}^{\varepsilon})$ to both sides of \eqref{g}, we obtain
	\begin{align}
		\partial_t B\left([g]_{x,q}^{\varepsilon}\right) + \nabla_x B\left([g]_{x,q}^{\varepsilon}\right) \cdot u &- \mathcal{L}B\left([g]_{x,q}^{\varepsilon}\right) + B^{''}\left([g]_{x,q}^{\varepsilon}\right)|\nabla_q [g]_{x,q}^{\varepsilon} |^2 ~~~~~~\\ \notag
		&= \left([h]_{x,q}^{\varepsilon} + s^{\varepsilon}+r^{\varepsilon}\right)B^{'}\left([g]_{x,q}^{\varepsilon}\right),
	\end{align}
	which yields \eqref{renormalizedeq2} for $B\in \mathscr{D}(0,\infty)$ by passing $\varepsilon \rightarrow 0$. For any given $B \in C^2[0,\infty)$ such that \eqref{renormalizedeq2} makes sense, take $B_n\in \mathscr{D}(0,\infty)$ which is uniformly bounded and converges to $B$ uniformly on compact sets in $[0,\infty)$. The proof of Propositon \ref{renormalized2} is completed by using Lebesgue's convergence theorem.
\end{proof}
\begin{lemm}\cite{Masmoudi2013}$(\rm Existence$ $\rm of$ $\rm Diperna$-$\rm Lions$ $\rm flow)$\label{Lions}
If $u\in L^2(0,T;H^1(\mathbb{R}^d))$ and ${\rm div}~u=0$. Then there exists a unique flow $X(t,t_0,x)$ such that for all $t_0\in(0,T)$ and $t\mapsto X(t,t_0,x)$ is absolutely continuous for $a.e.~x\in \mathbb{R}^d$ and satisfies
\begin{align}
\left\{
\begin{array}{ll}
\frac{\partial X}{\partial t}(t,t_0,x)=u(t,X(t,t_0,x)), \ t\in(0,T), \\
X(t_0,t_0,x)=x.
\end{array}
\right.
\end{align}
Moreover, for all $t,t_0\in(0,T)$, the map $x\mapsto X(t,t_0,x)$ is measure-preserving.
\end{lemm}
\begin{lemm}\cite{Masmoudi2013}$( \rm Mild$ $\rm formulation)$\label{Mild}
Assume that $u\in L^2\left(0,T;H^1_0(\Omega)\right)$ and that $X(t,x)$ is its $\rm Diperna$-$\rm Lions$ flow. Let $f\in L^{\infty}\left((0,T)\times\Omega\right),~f_0\in L^{\infty}(\Omega)$ and $h\in L^1\left((0,T)\times\Omega\right)$. The following three systems are equivalent :
\begin{align}
\left\{
\begin{array}{ll}
\partial_t f + u\cdot\nabla f \geq h $ in $ \mathscr{D}^{'}((0,T)\times\Omega), \\
f(t=0,x)\geq f_0(x),
\end{array}
\right.
\end{align}
\begin{align}
	~~~~
\left\{
\begin{array}{ll}
\frac{d}{dt}[f(t,X(t,x))] \geq h(t,X(t,x)) $ in $ \mathscr{D}^{'}((0,T)\times\Omega), \\
f(t=0,x)\geq f_0(x),
\end{array}
\right.
\end{align}
\begin{align}
	~~~~~
\left\{
\begin{array}{ll}
\frac{d}{dt}[f(t,X(t,x))] \geq h(t,X(t,x)) $ in $ \mathscr{D}^{'}((0,T)) $ for a.e. $ \ x \in\Omega, \\
f(t=0,x)\geq f_0(x).
\end{array}
\right.
\end{align}
In this case, we also have that $f(t,X(t,x))\in {\rm BV}\left(0,T;\mathcal{M}(\Omega)\right)$ and that for $a.e.~ x\in\Omega$, the function $f(t,X(t,x))\in {\rm BV}(0,T)$ and $h\left(t,X(t,x)\right)\in L^1(0,T)$.
\end{lemm}
\section{Priori estimates}
\begin{prop}\cite{Masmoudi2dfluid}\label{prop1}
Let $d\geq2$. Assume that $(u,g)$ is smooth solution of system \eqref{eq0} with $u_0\in L^2$, $\int_{\mathbb{R}^d}g_0\psi_{\infty}dq=0$ and $g_0 \in L^2(\mathcal{L}^2)$, then
\begin{align}\label{1ineq1}
\|u\|_{L^{\infty}_T(L^2)} + \|u\|_{L^{2}_T(\dot{H}^1)} \leq \|u_0\|_{L^2} + C\|g_0\|_{L^2(\mathcal{L}^2)},
\end{align}
and
\begin{align}\label{1ineq1'}
\|g\|_{L^{\infty}_T(L^{2}(\mathcal{L}^2))}+ \|\nabla_q g\|_{L^{2}_T (L^{2}(\mathcal{L}^2))}\leq \|g_0\|_{L^{2}(\mathcal{L}^2)}.~~~
\end{align}
Moreover, if $g_0 \in L^{p}(\mathcal{L}^2)$, then we obtain
\begin{align}\label{1ineq2}
\|g\|_{L^{\infty}_T(L^p(\mathcal{L}^2))} \leq \|g_0\|_{L^p(\mathcal{L}^2)},~~~~
\end{align}
for any $p\in [1,\infty]$.
\end{prop}

We establish additional weighted energy estimates in the following proposition.
\begin{prop}\label{prop2}
Let $d\geq2$ and $p\in [2,\infty]$. Assume $(u,g)$ is smooth solution of system \eqref{eq0} with $\int_{\mathbb{R}^d}g_0\psi_{\infty}dq=0$ and $\langle q\rangle g_0 \in L^{p}(\mathcal{L}^2)$. There exists a positive constant $C$ such that
\begin{align}\label{1ineq3}
\|\langle q\rangle g\|_{L^{\infty}_T(L^{p}(\mathcal{L}^2))} \leq \|\langle q\rangle g_0\|_{L^{p}(\mathcal{L}^2)}e^{CT}.
\end{align}
Moreover, if $p=2$, then there exists a positive constant $C$ such that
\begin{align}\label{1ineq3'}
~~~\|\langle q\rangle g\|^2_{L^{\infty}_T(L^{2}(\mathcal{L}^2))} + \int_0^T \|\langle q\rangle \nabla_q g\|^2_{L^{2}(\mathcal{L}^2)}dt \leq C\|\langle q\rangle g_0\|^2_{L^{2}(\mathcal{L}^2)}.
\end{align}
\end{prop}
\begin{proof}
Taking $\mathcal{L}^2$ inner product with $\langle q\rangle^2g\psi_{\infty}$ to system $\eqref{eq0}_2$, we infer that
\begin{align}\label{1ineq4}
\frac{1}{2}\frac{d}{dt}\|\langle q\rangle g\|^2_{\mathcal{L}^2} + \frac{1}{2}u\cdot\nabla\|\langle q\rangle g\|^2_{\mathcal{L}^2} - \int_{\mathbb{R}^d}\mathcal{L}g\cdot\psi_{\infty}g\langle q\rangle^2dq = \int_{\mathbb{R}^d} {\rm div}_q\left(\Omega q\cdot g\psi_{\infty}\right)\cdot\langle q\rangle^2gdq.
\end{align}
It follows that
\begin{align}\label{1ineq5}
~~\int_{\mathbb{R}^d} {\rm div}_q\left(\Omega q\cdot g\psi_{\infty}\right)\cdot\langle q\rangle^2gdq = 0,
\end{align}
and
\begin{align}\label{1ineq6}
-\int_{\mathbb{R}^d}\mathcal{L}g\cdot\psi_{\infty}g\langle q\rangle^2dq = \int_{\mathbb{R}^d}\langle q\rangle^2 |\nabla_qg|^2\psi_{\infty}dq + \int_{\mathbb{R}^d} |qg|^2\psi_{\infty}dq - d\int_{\mathbb{R}^d}g^2\psi_{\infty}dq.~~~
\end{align}
Combining \eqref{1ineq4}, \eqref{1ineq4} and \eqref{1ineq6}, we obtain
\begin{align}\label{1ineq7}
~~~~\frac{1}{2}\frac{d}{dt}\|\langle q\rangle g\|^2_{\mathcal{L}^2} + \frac{1}{2}u\cdot\nabla\|\langle q\rangle g\|^2_{\mathcal{L}^2} + \|\langle q\rangle\nabla_qg\|^2_{\mathcal{L}^2} + \|qg\|^2_{\mathcal{L}^2}  = d\|g\|^2_{\mathcal{L}^2},
\end{align}
which implies that
\begin{align}\label{1ineq8}
\frac{1}{2}\frac{d}{dt}\|\langle q\rangle g\|^2_{\mathcal{L}^2} + \frac{1}{2}u\cdot\nabla\|\langle q\rangle g\|^2_{\mathcal{L}^2} \leq d\|\langle q\rangle g\|^2_{\mathcal{L}^2}.
\end{align}
Multiplying $\|\langle q\rangle g\|^{p-2}_{\mathcal{L}^2}$ to \eqref{1ineq8}, we deduce that
\begin{align}\label{1ineq9}
\frac{1}{p}\frac{d}{dt}\|\langle q\rangle g\|^p_{\mathcal{L}^2} + \frac{1}{p}u\cdot\nabla\|\langle q\rangle g\|^p_{\mathcal{L}^2} \leq d\|\langle q\rangle g\|^p_{\mathcal{L}^2}.
\end{align}
Integrating over $\mathbb{R}^d$ with respect to $x$, we have
\begin{align}\label{1ineq10}
\frac{1}{p}\frac{d}{dt}\|\langle q\rangle g\|^p_{L^{p}(\mathcal{L}^2)} \leq d\|\langle q\rangle g\|^p_{L^{p}(\mathcal{L}^2)}.~~~~~~~~~~~~~~
\end{align}
Applying Gronwall's inequality to \eqref{1ineq10}, we conclude that
\begin{align}\label{1ineq11}
\|\langle q\rangle g\|_{L^{\infty}_T(L^{p}(\mathcal{L}^2))} \leq \|\langle q\rangle g_0\|_{L^{p}(\mathcal{L}^2)}e^{CT},~~~~~~~~~~
\end{align}
for any $p\in [2,\infty]$. Moreover, integrating \eqref{1ineq7} over $[0,T]\times\mathbb{R}^d$ and using \eqref{1ineq1}, we obtain
\begin{align}\label{1ineq12}
\|\langle q\rangle g\|^2_{L^{\infty}_TL^2(\mathcal{L}^2)}  + \int_0^T \|\langle q\rangle \nabla_q g\|^2_{L^2(\mathcal{L}^2)}dt \leq C\|\langle q\rangle g_0\|^2_{L^{2}(\mathcal{L}^2)}.~~~~
\end{align}
We thus complete the proof of Proposition \ref{prop2}.
\end{proof}

The following new energy estimates play a key role in the proof of the equi-integrability of $\|\nabla_qg\|^2_{\mathcal{L}^2}$.
\begin{prop}\label{Conservation}
Let $d\geq2$. Assume $(u,g)$ is smooth solution of system \eqref{eq0} with initial data $\nabla_q g_0 \in L^{p}(\mathcal{L}^2)$ for some $p\in [2,\infty]$. Let $\int_{\mathbb{R}^2} g_0\psi_{\infty}dq=0$. There exists a positive constant $C$ such that
\begin{align}\label{1ineq13}
\| \nabla_q g\|_{L^{\infty}_T(L^{p}(\mathcal{L}^2))} \leq \|\nabla_q g_0\|_{L^{p}(\mathcal{L}^2)}.~~~~~~~~~~~~~
\end{align}
Moreover, if $p=2$, then there exists a positive constant $C$ such that
\begin{align}\label{1ineq14}
\|\nabla_q g\|^2_{L^{\infty}_T(L^{2}(\mathcal{L}^2))} + \int_0^T \|\nabla^2_q g\|^2_{L^{2}(\mathcal{L}^2)}ds \leq \|\nabla_q g_0\|^2_{L^2(\mathcal{L}^2)}.~~~~~~~~~~~
\end{align}
\end{prop}
\begin{proof}
	Applying $\nabla_q$ to system $\eqref{eq0}_2$ and taking $\mathcal{L}^2$ inner product with $\nabla_qg\psi_{\infty}$ , we infer that
	\begin{align}\label{1ineq14'}
		~~\frac{1}{2}\frac{d}{dt}\|\nabla_q g\|^2_{\mathcal{L}^2} + \frac{1}{2}u\cdot\nabla\|\nabla_q g\|^2_{\mathcal{L}^2} &- \int_{\mathbb{R}^d}\nabla_q\mathcal{L}g\cdot\psi_{\infty}\nabla_qgdq \\ \notag
		&= \int_{\mathbb{R}^{d}}\nabla_q\left(\frac{1}{\psi_{\infty}}\nabla_q \cdot \big( \Omega q g \psi_{\infty}\big)\right)\nabla_qg\psi_{\infty}dq.~~~~~~~~~~~~~~~
	\end{align}
According to the antisymmetry of $\Omega$, one can arrive at
\begin{align}\nonumber
\int_{\mathbb{R}^{d}}\nabla_q\left(\frac{1}{\psi_{\infty}}\nabla_q \cdot \big( \Omega q g \psi_{\infty}\big)\right)\nabla_qg\psi_{\infty}dq &= \int_{\mathbb{R}^{d}}\nabla^l_q\big(\Omega^{ik} q_k\nabla^i_qg - \Omega^{ik} q_kq_ig\big)\nabla^l_qg\psi_{\infty}dq \\ \notag
&= \int_{\mathbb{R}^{d}}\big(\Omega^{ik} \delta^l_k\nabla^i_qg + \Omega^{ik} q_k\nabla^{il}_qg\big)\nabla^l_qg\psi_{\infty}dq \label{1ineq15} \\
&- \int_{\mathbb{R}^{d}}\big(\Omega^{ik} (\delta^l_kq_i+\delta^l_iq_k)g+\Omega^{ik} q_kq_i\nabla^l_qg\big)\nabla^l_qg\psi_{\infty}dq=0.
\end{align}
By virtue of Lemma \ref{Lemma1}, we deduce that
\begin{align}\nonumber
\int_{\mathbb{R}^{d}}\nabla_q \left(\frac{1}{\psi_{\infty}}\nabla_q \cdot\left(\nabla_q g\psi_{\infty}\right)\right)\nabla_qg\psi_{\infty}dq &=\int_{\mathbb{R}^{d}}\left(\frac{1}{\psi_{\infty}}\nabla_q\cdot\left(\nabla_q \nabla_q g\psi_{\infty}\right)-\nabla_qg\right)\nabla_qg\psi_{\infty}dq~~~~~~~~~\\ \notag
&=\int_{\mathbb{R}^{d}}\nabla_q\cdot\left(\nabla_q \nabla_q g\psi_{\infty}\right)\nabla_qgdq -\|\nabla_q g\|^2_{\mathcal{L}^2} \label{1ineq16}\\
&=-\|\nabla^2_q g\|^2_{\mathcal{L}^2} -\|\nabla_q g\|^2_{\mathcal{L}^2}.
\end{align}
We infer from system $\eqref{eq0}_2$, \eqref{1ineq15} and \eqref{1ineq16} that
\begin{align}\label{1ineq17}
\frac{1}{2}\frac{d}{dt}\|\nabla_q g\|^2_{\mathcal{L}^2} +\frac{1}{2} u\cdot\nabla\|\nabla_q g\|^2_{\mathcal{L}^2} + \|\nabla^2_q g\|^2_{\mathcal{L}^2}+\|\nabla_q g\|^2_{\mathcal{L}^2}=0.
\end{align}
Analogously, for any $p\geq2$, we conclude that
\begin{align}\label{1ineq18}
&\frac{1}{p}\frac{d}{dt}\|\nabla_q g\|^p_{L^p(\mathcal{L}^2)} \leq 0,~~~~~~~
\end{align}
which implies $\|\nabla_q g\|_{L^{\infty}_T(L^{p}(\mathcal{L}^2))} \leq \|\nabla_q g_0\|_{L^{p}(\mathcal{L}^2)}$. Moreover, integrating \eqref{1ineq17} over $[0,T]\times\mathbb{R}^d$, we obtain eventually
\begin{align}\label{1ineq19}
\|\nabla_q g\|^2_{L^{\infty}_T(L^2(\mathcal{L}^2))}  + \int_0^T \|\nabla^2_q g\|^2_{L^2(\mathcal{L}^2)}dt \leq \|\nabla_q g_0\|^2_{L^2(\mathcal{L}^2)}.~~~~~~~~~~
\end{align}
The proof of Proposition \ref{Conservation} is completed.
\end{proof}
\begin{lemm}\cite{lions1996}\label{prop3}
Let $d=3$. Assume $u\in L^{\infty}_T(L^2)\cap L^2_T(\dot{H}^1)$, then
\begin{align}
\|u\|_{L^{\alpha}_T(L^{\beta})}<\infty,
\end{align}
with $\alpha\in [2,\infty]$ and $\beta=\frac{6\alpha}{3\alpha-4}$. It also follows that
\begin{align}
\|u\otimes u\|_{L^{q}_T(L^{r})}, ~~~ \|u\nabla u\|_{L^{\gamma}_T(L^{\delta})}<\infty,
\end{align}
with $q\in[1,\infty]$, $r=\frac{3q}{3q-2}$ and $\gamma\in[1,2]$, $\delta=\frac{3\gamma}{4\gamma-2}$. Moreover, we have
\begin{align}
\|u\nabla u\|_{ L^{1}_T(L^{\frac{3}{2},1})} < \infty.~~~~~~
\end{align}
\end{lemm}

The following property is crucial to obtain the weak compactness of the velocity $u$.
\begin{prop}\cite{Bahouri2011}\label{prop4}
Let $d=3$. Assume ${\rm div}~u=0$ and $u\in L^{\infty}_T(L^2)\cap L^2_T(\dot{H}^1)$, then
\begin{align}
\Delta^{-1}{\rm div}(u\cdot\nabla u) \in L^1_T(\dot{W}^{2,1}) \cap L^2_T(\dot{W}^{1,1})\cap L^q_T(L^r),
\end{align}
with $q\in [1,\infty)$ and $r=\frac{3q}{3q-2}$.
\end{prop}
\begin{proof}
It follows that $ \partial_j(\partial_i u^j) = 0$ and $\nabla\times\nabla u^i =0$. According to Proposition \ref{prop0'}, we infer that
\begin{align}
\|\Delta^{-1}\nabla^2(\partial_iu^j\partial_ju^i)\|_{L^1_T(\mathcal{H}^1)} \leq C\|\partial_iu^j\partial_ju^i\|_{L^1_T(\mathcal{H}^1)} \leq C\|u\|^2_{L^2_T(\dot{H}^1)},
\end{align}
which implies that $\Delta^{-1}{\rm div}(u\cdot\nabla u) \in L^1_T(\dot{W}^{2,1})$. Similarly, we obtain
\begin{align}
\|\Delta^{-1}\nabla {\rm div}(u\nabla u)\|_{L^2_T(\mathcal{H}^1)}\leq C\|u\|_{L^{\infty}_T(L^2)}\|\nabla u\|_{L^2_T(L^2)}.~~~~~~~~~~~
\end{align}
This leads us to get that $\Delta^{-1}{\rm div}(u\cdot\nabla u) \in L^2_T(\dot{W}^{1,1})$. Moreover, applying Lemma \ref{prop3}, we deduce that
\begin{align}
\|\Delta^{-1}{\rm div}~{\rm div}(u\otimes u)\|_{L^q_T(L^r)} \leq C\|u\otimes u\|_{L^q_T(L^r)}.~~~~~~~~~~~~~~~~~~
\end{align}
Thus we have $\Delta^{-1}{\rm div}(u\cdot\nabla u) \in L^q_T(L^r)$.
\end{proof}
\section{Compactness}
\subsection{Compactness on the velocity u}
 In this section, we only consider the case $\Lambda=\mathbb{R}^d$ and $d=3$, since the other case is more easier. We firstly spilt $u$ into $u_1 + u_2 + u_3$, where $u_1,u_2$ and $u_3$ solve respectively
\begin{align}\label{eq1}
\left\{
\begin{array}{ll}
\partial_t u_1 -\Delta u_1 + \nabla P_1 = -{\rm div}~(u\otimes u),\\
{\rm div}~u_1=0,\ \ u_1|_{t=0}=0,
\end{array}
\right.
\end{align}
and
\begin{align}\label{eq2}
\left\{
\begin{array}{ll}
\partial_t u_2 -\Delta u_2 + \nabla P_2 = 0,\\
{\rm div}~u_2=0,\ \ u_2|_{t=0}=u_0,~~~~~~~~~~~~
\end{array}
\right.
\end{align}
where $P_2$ is a given harmonic function and
\begin{align}\label{eq3}
\left\{
\begin{array}{ll}
\partial_t u_3 -\Delta u_3 + \nabla P_3 = {\rm div}~\tau(g),~~~~~\\
{\rm div}~u_3=0,\ \ u_3|_{t=0}=0.
\end{array}
\right.
\end{align}

Let's recall the following compactness property.
\begin{lemm}\cite{lions1996,lions1998}\label{ucompact}
Assume $u^n \in C_T(L^2_w) \cap L^{\infty}_T(L^2)\cap L^2_T(\dot{H}^1)$, then there exists $u \in L^{\infty}_T(L^2)\cap L^2_T(\dot{H}^1)$ such that for any $K \subset\subset \mathbb{R}^3$,
$u^n \rightarrow u$ in $L^q_T(L^p(K))$ with $q\in [2,\infty)$ and $p<\frac{6q}{3q-4}$.
\end{lemm}

Then we introduce different compactness on velocity $u_i$ with $i=1,2,3.$ For more details, one can refer to \cite{lions1996,lions1998}.
\begin{prop}\label{2prop1}
Let $\{u^n\}_{n\in N}\in L^{\infty}_T(L^2)\cap L^2_T(\dot{H}^1)$ and $\{g^n\}_{n\in N} \in L^{\infty}_T(L^2(\mathcal{L}^2))$. Assume $u^n_1,u^n_2$ and $u^n_3$ solve systems \eqref{eq1}, \eqref{eq2} and \eqref{eq3}, respectively. Then there exist $u_1,u_2$ and $u_3$ such that
\begin{align}\label{2ineq1}
u^n_i \rightarrow u_i \in  L^q_T(L^p(K)), \ i=1,3 \ \  and \ \ \  u^n_2 \rightarrow u_2 \in L^{\infty}_T(L^2)\cap L^2_T(\dot{H}^1),
\end{align}
with $K\subset\subset\mathbb{R}^d, q\in [2,\infty)$ and $p<\frac{6q}{3q-4}$. Moreover,
\begin{align}\label{2ineq2}
\nabla u^n_1 \rightarrow \nabla u_1 \in L^{2}_T(L^{p_0}(K))\cap L^1_T(L^{r_0}(K)),~~
\end{align}
with $p_0<2$ and $r_0<3$.
\end{prop}
\begin{proof}
Using standard $L^2$ energy estimate to system \eqref{eq3} with $g^n \in L^{\infty}_T(L^2(\mathcal{L}^2))$, we obtain
\begin{align}\label{2ineq3}
u^n_3\in \dot{W}^{1,1}_T(L^2_w) \cap L^{\infty}_T(L^2)\cap L^2_T(\dot{H}^1) \ \ and \ \  \ P^n_3 = \Delta^{-1}{\rm div}~{\rm div}~\tau^n \in L^{\infty}_T(L^2).~~~~~
\end{align}
We infer from compact embedding theorem that there exists $u_3$ such that $u^n_3 \rightarrow u_3 \in C(0,T;W^{-\varepsilon,2}(K))$ for any positive $\varepsilon$. By virtue of interpolation inequality, we deduce that $u^n_3 \rightarrow u_3 \in  L^q_T(L^p(K))$ with $K\subset\subset\mathbb{R}^d, q\in [2,\infty)$ and $p<\frac{6q}{3q-4}$. According to system \eqref{eq2} and Duhamel's principle, we obtain
\begin{align}\label{2ineq4}
u^n_2 = e^{t\Delta}u^n_0 + \int_0^t e^{(t-s)\Delta}\nabla P^n_2 ds.
\end{align}
Taking harmonic functions $P^n_2 = [P_2]^{\frac 1 n}_x $ such that $P^n_2 \rightarrow P_2 \in H^2$, we deduce that there exists $u_2 \in L^{\infty}_T(L^2)\cap L^2_T(\dot{H}^1)$ such that $u^n_2 \rightarrow u_2 \in L^{\infty}_T(L^2)\cap L^2_T(\dot{H}^1) \hookrightarrow L^q_T(L^p(K))$. Notice that $u^n_1 = u^n - (u^n_2 + u^n_3)$, there exists $u_1 \in L^{\infty}_T(L^2)\cap L^2_T(\dot{H}^1)$ such that $u^n_1 \rightarrow u_1 \in L^q_T(L^p(K))$. Moreover, applying Leray projector $\mathbb{P}={\rm Id} - \Delta^{-1}\nabla~{\rm div}$ to system \eqref{eq1}, we infer from Duhamel's principle that
\begin{align}\label{2ineq5}
u^n_1 = \int_0^t e^{(t-s)\Delta} \mathbb{P}{\rm div}(u^n\otimes u^n) ds.
\end{align}
According to Lemma \ref{prop3}, we have ${\rm div}(u^n\otimes u^n) \in L^2_T(\mathcal{H}^1) \cap L^1_T(L^{\frac{3}{2},1}) \hookrightarrow L^k_T(L^r)$ with $k\in (1,2)$ and $r=\frac{3k}{4k-2}$. Then we deduce that $u_1^n \in L^k_T(\dot{W}^{2,r})$, which implies that $u_1^n \in L^{1}_T(\dot{W}^{1,r_0})$ with $r_0<3$. By virtue of compact embedding theorem and $u_1 \in  L^2_T(\dot{H}^1)$, we have
$\nabla u^n_1 \rightarrow \nabla u_1 \in L^{2}_T(L^{p_0}(K)) \cap L^1_T(L^{r_0}(K))$ with $p_0<2$ and $r_0< 3$. The proof of Proposition \ref{2prop1} is finished.
\end{proof}

Thus the compactness of $u_i$ with $i=1,2,3$ implies that $(u,g)$ satisfies system $\eqref{eq0}_1$ in the sense of Definition \ref{weakdefi}.

\subsection{Compactness on the polymeric distribution g}
According to the compactness of $u$ has been discussed in Proposition \ref{2prop1} and the boundness of $g$ in $L^2(\mathcal{L}^2)$, we infer that
$$
{\rm div}~(u^ng^n) \rightharpoonup {\rm div}~(ug) \in \mathscr{D}'([0,T]\times\mathbb{R}^3\times\mathbb{R}^3).
$$
The main difficulty is to prove that the following weak compactness holds :
\begin{align}\label{2ineq6}
\nabla_q\cdot(\Omega^nqg^n\psi_{\infty})\rightharpoonup\nabla_q\cdot(\Omega qg\psi_{\infty})  \in \mathscr{D}'([0,T]\times\mathbb{R}^3\times\mathbb{R}^3).
\end{align}
According to Proposition \ref{2prop1}, we infer that
$$
~\nabla_q\cdot[\nabla(u^n_1+u^n_2)g^n\psi_{\infty}]
 \rightharpoonup \nabla_q\cdot[\nabla(u_1+u_2)g\psi_{\infty}]\in \mathscr{D}^{'}([0,T)\times\mathbb{R}^3\times\mathbb{R}^3).
$$
Therefore, we are now concerned with $u_3$. We begin with the following equi-integrability for $\{g^n\}_{n\in N}$.
\begin{prop}\label{3prop1}
	Assume that $\{\nabla_qg^n\}_{n\in N}$ is bounded in $L^{2}_T(L^2(\mathcal{L}^2))$ and $\{g^n\}_{n\in N}$ is bounded in $L^{\infty}_T((L^2\cap L^p)(\mathcal{L}^2))$ for some $p>2$. Then $\{|g^n|^2\}_{n\in N}$ is equi-integrable in $L^1_T(L^1(\mathcal{L}^1))$.
\end{prop}
\begin{proof}
	Applying Lemma \ref{Lemma3} , we have
	\begin{align}\label{3ineq1}
		\|g^n\|_{L^{r}_T(L^{m}(\mathcal{L}^m))} \leq C\|g^n\|^{\frac{3}{m}-\frac{1}{2}}_{L^{\infty}_T(L^p(\mathcal{L}^2))}\|\nabla_qg^n\|^{\frac{2}{r}}_{L^2_T(L^2(\mathcal{L}^2))},~~~~
	\end{align}
	for $r=\frac{4m}{3(m-2)}$ and $m=\frac{10p-12}{3p-2}\in(2,\frac{10}{3})$. It follows from Chebyshev inequality that
	\begin{align}\label{3ineq2}
		\int_0^T\int_{\mathbb{R}^3}\int_{\mathbb{R}^3} |g^n|^2\psi_{\infty} 1_{\left\{|g^n|^2\psi_{\infty}\geq M\right\}} dqdxdt &\leq C_T\|g^n\|^2_{L^{r}_T(L^m(\mathcal{L}^m))}\|g^n\|^{2(1-\frac{2}{m})}_{L^{\infty}_T(L^2(\mathcal{L}^2))}M^{-(1-\frac{2}{m})},
	\end{align}
	which implies that
	\begin{align}\label{3ineq3}
		\int_0^T\int_{\mathbb{R}^3}\int_{\mathbb{R}^3} |g^n|^2 \psi_{\infty}1_{\left\{|g^n|^2\psi_{\infty}\geq M\right\}} dqdxdt \rightarrow 0 \ \ as \ \ \ M  \rightarrow \infty.~~~~
	\end{align}
	This completes the proof of Proposition \ref{3prop1} .
\end{proof}
\textbf{ The proof of Theorem \ref{th1} : }\\
According to Propositions \ref{prop1}, \ref{prop2} and \ref{3prop1} , we introduce the following defect measures :
\begin{align}\label{3ineq7}
\left\{
\begin{array}{ll}
|\nabla(u^n_3-u_3)|^2 \rightharpoonup \mu \in \mathcal{M}(x), \\
|g^n-g|^2 \rightharpoonup \eta \in L^{\infty}_T((L^1\cap L^{\infty})(\mathcal{L}^1)),\\
|\tau^n-\tau|^2\rightharpoonup \alpha \in L^{\infty}_T(L^1\cap L^{\infty}) $ with $ \alpha \leq \int_{\mathbb{R}^d}\eta\psi_{\infty}dq ,  \\
\psi_{\infty}|\nabla_q(g^n-g)|^2 \rightharpoonup \kappa \in \mathcal{M}(x,q),\\
\langle q\rangle g^n\nabla u^n_3 \rightharpoonup \langle q\rangle g\nabla u_3 + \beta \in L^2_T(L^1(\mathcal{L}^2)) $ with $ |\beta|\leq \langle q\rangle\sqrt{\mu}\sqrt{\eta},~~~~~~~~~~
\end{array}
\right.
\end{align}
where $\mathcal{M}(\cdot)$ and $\mathcal{M}(\cdot,\cdot)$ are the spaces of bounded measure on $\mathbb{R}^3$ and $\mathbb{R}^3 \times \mathbb{R}^3$, respectively. Note that all measure inequalities hold in the sense of almost everywhere. For simplify, we omit the notion a.e. here.
Recalling that $u^n_3$ and $u_3$ solve the system \eqref{eq3}, we obtain
\begin{align}\label{3ineq8}
\left\{
\begin{array}{ll}
\frac{1}{2}\partial_t|u^n_3|^2 - \frac{1}{2}\Delta \left(|u^n_3|^2\right) + |\nabla u^n_3|^2 + {\rm div}~(u^n_3P^n_3) = {\rm div}(u^n_3\tau^n) - Tr\left(\tau^n(\nabla u^n_3)^{t}\right), \\
\frac{1}{2}\partial_t|u_3|^2 - \frac{1}{2}\Delta \left(|u_3|^2\right) + |\nabla u_3|^2 + {\rm div}~(u_3P_3) = {\rm div}(u_3\tau) - Tr\left(\tau(\nabla u_3)^{t}\right).
\end{array}
\right.
\end{align}
Passing to the limit in \eqref{3ineq8} and using the convergence properties already shown in \eqref{3ineq7}, we
obtain
\begin{align}\label{3ineq9}
~~~\frac{1}{2}\partial_t|u_3|^2 - \frac{1}{2}\Delta \left(|u_3|^2\right) + |\nabla u_3|^2 + \mu + {\rm div}~(u_3P_3) &= {\rm div}(u_3\tau) - Tr\left(\tau(\nabla u_3)^{t}\right)~~~~~ \\ \notag
&~~~~-\int_{\mathbb{R}^d}\beta_{ji}\frac{q_i\nabla_j\mathcal{U}}{\langle q\rangle}\psi_{\infty} dq.
\end{align}
Thus \eqref{3ineq9} minus \eqref{3ineq8} leads to
\begin{align}\label{3ineq10}
\mu = - \int_{\mathbb{R}^d}\beta_{ji}\frac{q_i\nabla_j\mathcal{U}}{\langle q\rangle}\psi_{\infty} dq,~~~~~~
\end{align}
which implies that $\mu \in L^2_T(L^1)$. According to \eqref{3ineq7}, we infer that
\begin{align}\label{3ineq11}
\int_{\mathbb{R}^3}\beta_{ji}\frac{q_i\nabla_j\mathcal{U}}{\langle q\rangle} \psi_{\infty}dq \leq C\sqrt{\mu}\sqrt{\alpha}.~~~~~
\end{align}
This gives $\mu \leq C\alpha\leq C\int_{\mathbb{R}^3}\eta\psi_{\infty}dq.$ Therefore, we conclude that
\begin{align}\label{3ineq12}
\left(\int_{\mathbb{R}^3}\frac{|\beta|^2}{\langle q\rangle^2}\psi_{\infty} dq\right)^{\frac{1}{2}} \leq C\sqrt{\mu}\left(\int_{\mathbb{R}^3}\eta\psi_{\infty}dq\right)^{\frac 1 2} \leq C\int_{\mathbb{R}^3}\eta\psi_{\infty}dq.~~~~~~~~~
\end{align}
Taking $\mathcal{L}^2$ inner product with $g^n\psi_{\infty}$ to system $\eqref{eq0}_2$, we infer that
\begin{align}\label{3ineq13}
\partial_t \|g^n\|^2_{\mathcal{L}^2} + {\rm div}~(u^n\|g^n\|^2_{\mathcal{L}^2})+ 2\|\nabla_qg^n\|^2_{\mathcal{L}^2}= 0.~~~~~~~
\end{align}
Passing to the limit in \eqref{3ineq13} and using \eqref{3ineq7}, we
obtain
\begin{align}\label{3ineq14}
\partial_t \left(\|g\|^2_{\mathcal{L}^2}+\int_{\mathbb{R}^3}\eta\psi_{\infty}dq\right)+ {\rm div}~\left(u\left(\|g\|^2_{\mathcal{L}^2}+\int_{\mathbb{R}^3}\eta\psi_{\infty}dq\right)\right)+ 2\|\nabla_qg\|^2_{\mathcal{L}^2}+2\|\kappa\|_{\mathcal{M}(q)}= 0.
\end{align}
According to system \eqref{eq0}, Proposition \ref{renormalized2} and \eqref{3ineq7}, we deduce that
\begin{align}\label{3ineq15}
\partial_t \|g\|^2_{\mathcal{L}^2} + {\rm div}~\left(u\|g\|^2_{\mathcal{L}^2}\right)+2\|\nabla_qg\|^2_{\mathcal{L}^2} = \int_{\mathbb{R}^3}\left(\beta_{ij}-\beta_{ji}\right)\frac{q_j}{\langle q\rangle}\nabla^i_qg\psi_{\infty}dq.~~~~
\end{align}
By using \eqref{3ineq12}, \eqref{3ineq14}, \eqref{3ineq15} and Propositions \ref{prop1}, \ref{prop2}, one can arrive at
\begin{align}\label{3ineq16}
\partial_t \int_{\mathbb{R}^3}\eta\psi_{\infty}dq + {\rm div}\left(u\int_{\mathbb{R}^3}\eta\psi_{\infty}dq\right)+ 2\|\kappa\|_{\mathcal{M}(q)}
= \int_{\mathbb{R}^3}\left(\beta_{ij}-\beta_{ji}\right)\frac{q_j}{\langle q\rangle}\nabla^i_q g\psi_{\infty}dq \in L^1_T(L^1_x).~~
\end{align}
We infer from Lemmas \ref{Lions} and \ref{Mild} that
\begin{align}\label{3ineq16'}
\partial_t \int_{\mathbb{R}^3}\eta\psi_{\infty}dq + {\rm div}\left(u\int_{\mathbb{R}^3}\eta\psi_{\infty}dq\right)\leq C\|\langle q\rangle\nabla_q g\|_{\mathcal{L}^2} \int_{\mathbb{R}^3}\eta\psi_{\infty}dq \ \ a.e.~ x\in\mathbb{R}^3 \ ,~~~~~~~~~~~~~~~~~~
\end{align}
and thus
\begin{align}\label{3ineq17}
~~~~\int_{\mathbb{R}^3}\eta\psi_{\infty}dq(X(t,x)) \leq \int_{\mathbb{R}^3}\eta_0\psi_{\infty}dq\cdot e^{C\int_0^t\|\langle q\rangle\nabla_qg\|_{\mathcal{L}^2}ds} \ a.e.~ x\in\mathbb{R}^3 ,~~
\end{align}
where $t\geq 0$ and $X$ is the unique $a.e.~ x\in\mathbb{R}^3$ flow such that
\begin{align}
 \dot{X}=u(t,X),  \ \ \ \ X(0,x)=x.~~
\end{align}
For each $t\in[0,T]$, according to Proposition \ref{prop2} and Minkowski's inequality, we deduce that
\begin{align}\label{3ineq18}
\|\int_0^t\|\langle q\rangle\nabla_qg\|_{\mathcal{L}^2}ds\|_{L^2}\leq \int_0^t\|\langle q\rangle\nabla_qg\|_{L^2(\mathcal{L}^2)}ds<\infty.~~~
\end{align}
Then $e^{C\int_0^t\|\langle q\rangle\nabla_qg\|_{\mathcal{L}^2}ds}<\infty$ implies $\int_{\mathbb{R}^d}\eta\psi_{\infty}dq(X(t,x))=0$ $a.e.~ x\in\mathbb{R}^3$ with $\int_{\mathbb{R}^d}\eta_0\psi_{\infty}dq=0$. Using the invariance of Lebesgue measure, we conclude that $\eta=0$ $a.e.~ (x,q)\in\mathbb{R}^3\times\mathbb{R}^3$ for all $t \geq 0$ and hence $g^n$ converges strongly to $g$. This completes the proof of Theorem \ref{th1}.
\hfill$\Box$

In the process of constructing renormalization equation shown in Theorem \ref{th2}, the main difficulty is to prove that the sum of measures produced by $|\nabla_qg^n|^2$ is positive with lower integrability of $\nabla_qg^n$. Thanks to the new energy estimates in Proposition \ref{Conservation}, we obtain the equi-integrability of  $|g^n|^2$ and $|\nabla_qg^n|^2$.
\begin{prop}\label{3prop2}
	Assume that $\{\nabla_qg^n\}_{n\in N}$ is bounded in $L^{\infty}_T((L^2\cap L^p)(\mathcal{L}^2))$ for some $p>2$ and $\{\nabla^2_qg^n\}_{n\in N}$ is bounded in $L^{2}_T(L^2(\mathcal{L}^2))$, then $\{|\nabla_qg^n|^2\}_{n\in N}$ is equi-integrable in $L^1_T(L^1(\mathcal{L}^1))$.
\end{prop}
\begin{proof}
	Taking $m=\frac{10p-12}{3p-2}$ with $m\in(2,\frac{10}{3}]$, we deduce that
	\begin{align}\label{3ineq4}
		~~\|\nabla_qg^n\|_{L^{\frac{4m}{3(m-2)}}_T(L^{m}(\mathcal{L}^{m}))} &\leq C\left(\int_0^T\left(\int_{\mathbb{R}^3} \|\nabla_qg^n\|^{3-\frac{m}{2}}_{\mathcal{L}^2}\|\nabla^2_qg^n\|^{\frac{3m}{2}-3}_{\mathcal{L}^2}dx\right)^{\frac{4}{3(m-2)}}dt\right)^{\frac{3(m-2)}{4m}} \\ \notag
		&\leq C\|\nabla_qg^n\|^{\frac{3}{m}-\frac{1}{2}}_{L^{\infty}_T(L^p(\mathcal{L}^2))}\|\nabla^2_qg^n\|^{\frac{3}{2}-\frac{3}{m}}_{L^2_T(L^2(\mathcal{L}^2))},
	\end{align}
	which implies that $\nabla_qg^n\in L^{m_1}_T(L^{m}(\mathcal{L}^{m}))$ with $m_1=\frac{4m}{3(m-2)}>2$.  By virtue of Chebyshev inequality, one can arrive at
	\begin{align}\label{3ineq5}
		\iint_{\{|\nabla_qg^n|^2 \psi_{\infty}\geq M\}}|\nabla_qg^n|^2 \psi_{\infty}dqdx \leq C_T \|\nabla_qg^n\|^{2-\frac{4}{m}}_{L^2(\mathcal{L}^2)} \|\nabla_qg^n\|^2_{L^m(\mathcal{L}^m)}M^{-1+\frac{2}{m}},
	\end{align}
	which implies that
	\begin{align}\label{3ineq6}
		\int_0^T\iint_{\{|\nabla_qg^n|^2 \psi_{\infty}\geq M\}}|\nabla_qg^n|^2 \psi_{\infty}dqdxdt \rightarrow 0 \ \ as \ \ \ M  \rightarrow \infty.
	\end{align}
	This completes the proof of Proposition \ref{3prop2} .
\end{proof}

\begin{lemm}\label{3lemm1}
Let the conditions in Proposition \ref{3prop2} be fulfilled. Suppose that
\begin{align}\label{3ineq19}
\|\nabla_qg^n\|^2_{\mathcal{L}^2} \rightharpoonup \|\nabla_qg\|^2_{\mathcal{L}^2} + \tilde{\kappa},
\end{align}
and
\begin{align}\label{3ineq20}
~\frac{\|\nabla_qg^n\|^2_{\mathcal{L}^2}}{(1+\delta\| g^n\|^2_{\mathcal{L}^2})^2} \rightharpoonup \frac{\|\nabla_qg\|^2_{\mathcal{L}^2}}{(1+\delta\| g\|^2_{\mathcal{L}^2})^2} + \tilde{\kappa}_{\delta},
\end{align}
with $\tilde{\kappa},\tilde{\kappa}_{\delta}\in L^{\infty}_T(L^1\cap L^{\frac p 2})$ for some $p>2$, then
\begin{align}\label{3ineq21}
\tilde{\kappa}_{\delta} \rightharpoonup \tilde{\kappa} \in L^{\infty}_T(L^1) \ \ as \ \ \ \delta \rightarrow 0.
\end{align}
\end{lemm}
\begin{proof}
It's sufficient to prove that $\left|\frac{\|\nabla_qg^n\|^2_{\mathcal{L}^2}}{\left(1+\delta\| g^n\|^2_{\mathcal{L}^2}\right)^2}-\|\nabla_qg^n\|^2_{\mathcal{L}^2}\right|\rightarrow0\in L^{\infty}_T(L^1)$ as $\delta\rightarrow0$. Firstly, we deduce that
\begin{align}\label{3ineq22}
\left|\frac{\|\nabla_qg^n\|^2_{\mathcal{L}^2}}{\left(1+\delta\| g^n\|^2_{\mathcal{L}^2}\right)^2}-\|\nabla_qg^n\|^2_{\mathcal{L}^2}\right|&=\frac{\|\nabla_qg^n\|^2_{\mathcal{L}^2}}{\left(1+\delta\| g^n\|^2_{\mathcal{L}^2}\right)^2}\left(\delta^2\| g^n\|^4_{\mathcal{L}^2}+2\delta\| g^n\|^2_{\mathcal{L}^2}\right)\\ \notag
& \leq \|\nabla_qg^n\|^2_{\mathcal{L}^2}1_{\{\|\nabla_qg^n\|^2_{\mathcal{L}^2}>M\}} +3M\delta\| g^n\|^2_{L^2(\mathcal{L}^2)}.
\end{align}
Applying Chebyshev inequality, one can arrive at
\begin{align}\label{3ineq23}
\int_{\mathbb{R}^3}\|\nabla_qg^n\|^2_{\mathcal{L}^2}1_{\{\|\nabla_qg^n\|^2_{\mathcal{L}^2}>M\}}dx &\leq \|\nabla_qg^n\|^2_{L^p(\mathcal{L}^2)}\|\nabla_qg^n\|^{2-\frac{4}{p}}_{L^2(\mathcal{L}^2)}M^{\frac{2}{p}-1}.~~
\end{align}
This together with \eqref{3ineq22} implies that
\begin{align}\label{3ineq24}
\sup_{t\in[0,T]}\int_{\mathbb{R}^3}\left|\frac{\|\nabla_qg^n\|^2_{\mathcal{L}^2}}{\left(1+\delta\| g^n\|^2_{\mathcal{L}^2}\right)^2}-\|\nabla_qg^n\|^2_{\mathcal{L}^2}\right|dx\rightarrow0 \ \ as \ \ \ \delta \rightarrow0.
\end{align}
We thus complete the proof of Lemma \ref{3lemm1} .
\end{proof}
\textbf{The proof of Theorem \ref{th2} : } \\
According to Propositions \ref{prop1}, \ref{prop2}, \ref{Conservation}, \ref{3prop1} and \ref{3prop2}, we obtain the following defect measures:
\begin{align}\label{4ineq1}
\left\{
\begin{array}{ll}
|\nabla\left(u^n_3-u_3\right)|^2 \rightharpoonup \mu \in \mathcal{M}(x), \\
|g^n-g|^2 \rightharpoonup \eta \in L^{\infty}_T\left((L^1\cap L^{\frac{p}{2}})(\mathcal{L}^1)\right),\\
|\tau^n-\tau|^2\rightharpoonup \alpha \in L^{\infty}_T\left(L^1\cap L^{\frac{p}{2}}\right) $ with $ \alpha \leq \int_{\mathbb{R}^d}\eta\psi_{\infty}dq,  \\
|\nabla_q\left(g^n-g\right)|^2 \rightharpoonup \kappa \in L^{\frac{2m}{3(m-2)}}_T\left((L^1\cap L^{\frac{m}{2}})(\mathcal{L}^{\frac{m}{2}})\right) $ with $ m=\frac{10p-12}{3p-2},\\
\langle q\rangle g^n\nabla u^n_3 \rightharpoonup \langle q\rangle g^n\nabla u_3 + \beta \in L_T^2\left(L^1(\mathcal{L}^2)\right) $ with $ |\beta|\leq \langle q\rangle\sqrt{\mu}\sqrt{\eta}.
\end{array}
\right.
\end{align}
Extracting subsequences if necessary, for each $\delta\in(0,1)$, we assume that
\begin{align}\label{4ineq2}
\frac{\|g^n\|^2_{\mathcal{L}^2}}{1+\delta\|g^n\|^2_{\mathcal{L}^2}} \rightharpoonup \frac{\|g\|^2_{\mathcal{L}^2}}{1+\delta\|g\|^2_{\mathcal{L}^2}} +  \tilde{\eta}_{\delta},~~0\leq \tilde{\eta}_{\delta} \leq \frac{1}{\delta},~~~~~~~~~
\end{align}
\begin{align}\label{4ineq3}
\frac{\|\nabla_qg^n\|^2_{\mathcal{L}^2}}{\left(1+\delta\|g^n\|^2_{\mathcal{L}^2}\right)^2} \rightharpoonup \frac{\|\nabla_qg\|^2_{\mathcal{L}^2}}{\left(1+\delta\|g\|^2_{\mathcal{L}^2}\right)^2} + \tilde{\kappa}_{\delta},~~\tilde{\kappa}_{\delta}\in L^{\frac{2m}{3(m-2)}}_T(L^{\frac{m}{2}}),
\end{align}
\begin{align}\label{4ineq4}
~~~~~~\|g^n\|^2_{\mathcal{L}^2} + 1 \rightharpoonup N^2 \ with \ N=\sqrt{\|g\|^2_{\mathcal{L}^2} + \int_{\mathbb{R}^3}\eta\psi_{\infty}dq + 1}.
\end{align}
Taking $\mathcal{L}^2$ inner product with $g^n\psi_{\infty}$ to system $\eqref{eq0}_2$, we infer that
\begin{align}\label{4ineq5}
\partial_t \|g^n\|^2_{\mathcal{L}^2} + {\rm div}~\left(u^n\|g^n\|^2_{\mathcal{L}^2}\right)+2\|\nabla_qg^n\|^2_{\mathcal{L}^2} = 0.~
\end{align}
Denote that $\tilde{\eta}=\int_{\mathbb{R}^3}\eta\psi_{\infty}dq$ and $\tilde{\kappa}=\int_{\mathbb{R}^3}\kappa\psi_{\infty}dq$. Passing to the limit in \eqref{4ineq5} and using the convergence properties in \eqref{4ineq1} and \eqref{4ineq4}, one can arrive at
\begin{align}\label{4ineq6}
~~~~\partial_t N + {\rm div}~(uN)+\frac{1}{N}\left(\int_{\mathbb{R}^3}|\nabla_q g|^2\psi_{\infty} dq+\tilde{\kappa}\right) = 0.
\end{align}
Multiplying $\left(1+\delta\|g^n\|^2_{\mathcal{L}^2}\right)^{-2}$ to \eqref{4ineq5}, we infer that
\begin{align}\label{4ineq7}
\partial_t \frac{\|g^n\|^2_{\mathcal{L}^2}}{1+\delta\|g^n\|^2_{\mathcal{L}^2}} + {\rm div}~\left(u\frac{\|g^n\|^2_{\mathcal{L}^2}}{1+\delta\|g^n\|^2_{\mathcal{L}^2}}\right)+\frac{2\|\nabla_qg^n\|^2_{\mathcal{L}^2}}{\left(1+\delta\|g^n\|^2_{\mathcal{L}^2}\right)^2} = 0.~~~~~~~
\end{align}
Passing to the limit in \eqref{4ineq7} and using \eqref{4ineq2}, \eqref{4ineq3}, we deduce that
\begin{align}\label{4ineq8}
\partial_t \left(\frac{\|g\|^2_{\mathcal{L}^2}}{1+\delta\|g\|^2_{\mathcal{L}^2}} +  \tilde{\eta}_{\delta}\right) + {\rm div}~\left(u\left(\frac{\|g\|^2_{\mathcal{L}^2}}{1+\delta\|g\|^2_{\mathcal{L}^2}} +  \tilde{\eta}_{\delta}\right)\right)+2\frac{\|\nabla_qg\|^2_{\mathcal{L}^2}}{\left(1+\delta\|g\|^2_{\mathcal{L}^2}\right)^2} +  2\tilde{\kappa}_{\delta} = 0.~~~~~
\end{align}
According to Proposition \ref{renormalized2} , we obtain
\begin{align}\label{4ineq9}
	\partial_t\frac{\|g\|^2_{\mathcal{L}^2}}{1+\delta\|g\|^2_{\mathcal{L}^2}} &+ {\rm div}~\left(u\frac{\|g\|^2_{\mathcal{L}^2}}{1+\delta\|g\|^2_{\mathcal{L}^2}} \right)+2\frac{\|\nabla_qg\|^2_{\mathcal{L}^2}}{\left(1+\delta\|g\|^2_{\mathcal{L}^2}\right)^2}  \\ \notag
	&=\frac{1}{\left(1+\delta\|g\|^2_{\mathcal{L}^2}\right)^2}\int_{\mathbb{R}^3}\left(\beta_{ji}-\beta_{ij}\right)\frac{q_j}{\langle q\rangle}\nabla^i_qg\psi_{\infty}dq.~~~~~
\end{align}
Therefore, \eqref{4ineq8} minus \eqref{4ineq9} leads to
\begin{align}\label{4ineq10}
\partial_t \tilde{\eta}_{\delta} + {\rm div}~(u\tilde{\eta}_{\delta}) + 2\tilde{\kappa}_{\delta} = \frac{1}{(1+\delta\|g\|^2_{\mathcal{L}^2})^2}\int_{\mathbb{R}^3}(\beta_{ji}-\beta_{ij})\frac{q_j}{\langle q\rangle}\nabla^i_qg\psi_{\infty}dq.
\end{align}
Combining \eqref{4ineq6} and \eqref{4ineq10}, we conclude that
\begin{align}\label{4ineq11}
	\partial_t \frac{\tilde{\eta}_{\delta}}{N^2} + {\rm div}~(u\frac{\tilde{\eta}_{\delta}}{N^2}) + \frac{2}{N^2}\tilde{\kappa}_{\delta}-\frac{2\tilde{\eta}_{\delta}}{N^4}\tilde{\kappa} &= \frac{(1+\delta\|g\|^2_{\mathcal{L}^2})^{-2}}{N^{2}}\int_{\mathbb{R}^3}(\beta_{ji}-\beta_{ij})\frac{q_j}{\langle q\rangle}\nabla^i_qg\psi_{\infty} dq ~~~\\ \notag
	& ~~~~+\frac{2\tilde{\eta}_{\delta}}{N^4}\int_{\mathbb{R}^3}|\nabla_qg|^2\psi_{\infty} dq.
\end{align}
According to Lemma \ref{3lemm1} , passing $\delta$ to $0$ leads to
\begin{align}\label{4ineq12}
	\partial_t \frac{\tilde{\eta}}{N^2} + {\rm div}~(u\frac{\tilde{\eta}}{N^2}) + 2\left(\frac{1}{N^2}-\frac{\tilde{\eta}}{N^4}\right)\tilde{\kappa}&= \frac{1}{N^2}\int_{\mathbb{R}^3}(\beta_{ji}-\beta_{ij})\frac{q_j}{ \langle q\rangle}\nabla^i_qg\psi_{\infty}  dq ~~~~~~~~~~~~~~~ \\ \notag
	& ~~~~+\frac{2\tilde{\eta}}{N^4}\int_{\mathbb{R}^3}|\nabla_qg|^2\psi_{\infty} dq.
\end{align}
Notice that $\frac{1}{N^2}-\frac{\tilde{\eta}}{N^4} \geq 0$ and $0 \leq \frac{1}{N} \leq 1$. According to \eqref{3ineq12} and Proposition \ref{prop2}, we infer that
\begin{align}\label{4ineq13}
	~~\partial_t \frac{\tilde{\eta}}{N^2} + {\rm div}~(u\frac{\tilde{\eta}}{N^2}) \leq  C\left(\|\langle q\rangle\nabla_qg\|_{\mathcal{L}^2}+\frac {\|\langle q\rangle\nabla_qg\|^2_{\mathcal{L}^2}} {N^2}\right)\frac{\tilde{\eta}}{N^2}\in L^1_T(L^1_{\rm loc}).~~~~~~~~~~~~~~~~~~~~~~~~
\end{align}
According to Lemmas \ref{Lions} and \ref{Mild} with $\frac{\tilde{\eta}}{N^2} \in L^{\infty}_T(L^{\infty})$, we conclude for any $t\geq 0$ that
\begin{align}\label{4ineq14}
	~~~~~~~~~~~\frac{\tilde{\eta}}{N^2}(X(t,x)) \leq \frac{\tilde{\eta}_0}{\|g_0\|^2_{\mathcal{L}^2} + \tilde{\eta}_0 + 1}e^{Ct+C\int_0^t\|\langle q\rangle\nabla_qg_0\|^2_{\mathcal{L}^2}ds} \ a.e.~x\in\mathbb{R}^3 \ ,~~~~~~~~~~~~~~~~~~~
\end{align}
where $X$ is the unique $a.e.~x\in\mathbb{R}^3$ flow such that
\begin{align}
	\dot{X}=u(t,X), \ \ \ \ X(0,x)=x.~~~~
\end{align}
For each $t\in[0,T]$, it follows from Proposition \ref{prop2} and Minkowski's inequality that
\begin{align}\label{4ineq15}
	~~\|\int_0^t\|\langle q\rangle\nabla_qg\|^2_{\mathcal{L}^2}ds\|_{L^1} \leq \int_0^t\|\langle q\rangle\nabla_qg\|^2_{L^2(\mathcal{L}^2)}ds<\infty,
\end{align}
which implies that $e^{C\int_0^t\|\langle q\rangle\nabla_qg\|^2_{\mathcal{L}^2}ds}<\infty \ a.e.~x\in\mathbb{R}^3 $ and thus $\frac{\tilde{\eta}}{N^2}(X(t,x))=0$ $a.e.~x\in\mathbb{R}^3$ with $\tilde{\eta}_0=0$. Using the invariance of Lebesgue measure, we deduce that $\eta=0$ $a.e.~(x,q)\in\mathbb{R}^3\times\mathbb{R}^3$ for all $t \geq 0$ and thus complete the proof of Theorem \ref{th2}.
\hfill$\Box$
\begin{rema}
	In Theorems \ref{th1} and \ref{th2}, we obtain global existence of system \eqref{eq0} with additional energy estimates. Global existence of system \eqref{eq0} with standard energy estimations in \eqref{1ineq1} is an interesting problem. However, the technique in this paper fails to solve this problem and we would get further research in the furture.
\end{rema}
\section{Optimal decay rate}
\textbf{The proof of Theorem \ref{th3} : } \\
By density argument, we assume that $(u,g)$ is smooth solution of system $\eqref{eq0}$. We first prove the exponential decay rate of $g$ in $L^2(\mathcal{L}^2)$. Taking $L^2(\mathcal{L}^2)$ inner product with $g$ to system $(\ref{eq0})_2$, we infer
\begin{align}\label{5ineq1}
	\frac{d}{dt}\|g\|^2_{L^2(\mathcal{L}^2)} +  2\|\nabla_qg\|^2_{L^2(\mathcal{L}^2)} = 0.
\end{align}
According to Lemma \ref{Lemma1}, we have
\begin{align}\label{5ineq2}
	c\|g\|^2_{\mathcal{L}^2} \leq \|\nabla_qg\|^2_{\mathcal{L}^2},~~~~~
\end{align}
which implies that
\begin{align}\label{5ineq2'}
\|g\|^2_{L^2(\mathcal{L}^2)} \leq \|g_0\|^2_{L^2(\mathcal{L}^2)}e^{-2ct}.
\end{align}
Applying H\"{o}lder inequality, we obtain
\begin{align}\label{5ineq2''}
	~~~~~~~~~~~~~\|\tau\|^2_{L^2} \leq C\|g\|^2_{L^2(\mathcal{L}^2)}\leq Ce^{-2ct}.
\end{align}

Then we prove the optimal $L^2$ decay rate for velocity $u$. The proof is divided into three steps. To start with, we get initial time decay rate $\ln^{-l}(e+t)$ for $u$ in $L^2$ for any $l\in N^{+}$ by the Fourier splitting method. Then, by virtue of the time weighted energy estimate and the logarithmic decay rate, we improve the time decay rate to $(1+t)^{-\frac{1}{2}}$. Finally, we establish the lower bound of long time decay rate in $L^2$ for velocity $u$, which implies that the decay rate we obtained is optimal. \\
\textbf{Step 1 : }  Taking $L^2$ energy estimate to $\eqref{eq0}_1$, we deduce that
\begin{align}\label{5ineq3}
	\frac d {dt} \|u\|^2_{L^2}+\|\nabla u\|^2_{L^2}
	\leq  C\|\tau\|^2_{L^2}.
\end{align}
Define $S_0(t)=\{\xi:|\xi|^2\leq C_d\frac {f'(t)} {f(t)}\}$ with $f(t)=\ln^3(e+t)$ and the constant $C_d$ large enough. According to Schonbek's strategy, we have
\begin{align}\label{5ineq3'}
C_d\frac {f'(t)} {f(t)} \int_{(S_0(t))^c}|\hat{u}(\xi)|^2 d\xi\leq\|\nabla u\|^2_{L^2}.~~~~~
\end{align}
	We infer from \eqref{5ineq2''} and \eqref{5ineq3'} that
	\begin{align}\label{5ineq3''}
		\frac d {dt} \|u\|^2_{L^2}+C_d\frac {f'(t)} {f(t)}\| u\|^2_{L^2}
		\leq C_d\frac {f'(t)} {f(t)}\int_{S_0(t)}|\hat{u}(\xi)|^2 d\xi + Ce^{-2ct}.~~~~
	\end{align}
	Taking Fourier transform with respect to $x$ in system $(\ref{eq0})$, one can arrive at
	\begin{align}\label{5ineq4}
		\left\{
		\begin{array}{ll}
			\hat{u}_t+\mathscr{F}\left(u\cdot\nabla u\right)+|\xi|^2 \hat{u}+i\xi\hat{P}=i\xi\cdot\hat{\tau},  \\[1ex]
			\hat{g}_t+\mathscr{F}\left(u\cdot\nabla g\right)-\mathcal{L}(\hat{g})={\rm div}_{q}\left(-\mathscr{F}\left(\Omega\cdot{q}g\psi_{\infty}\right)\right), ~~\\[1ex]
			i\xi\cdot\bar{\hat{u}}=-\overline{i\xi\cdot\hat{u}}=0.
		\end{array}
		\right.
	\end{align}
	Multiplying $\bar{\hat{u}}$ to system $(\ref{5ineq4})_1$, we get
	\begin{align}\label{5ineq5}
    ~\partial_t|\hat{u}|^2 \leq |\hat{\tau}|^2 + C\left|\mathscr{F}(u\otimes u)\right|^2.
	\end{align}
    Taking $\mathcal{L}^2$ inner product to system  $(\ref{5ineq4})_2$ with $\hat{g}$, we obtain
    \begin{align}\label{5ineq6}
    	\partial_t\|\hat{g}\|^2_{\mathcal{L}^2} + 2\|\nabla_q\hat{g}\|^2_{\mathcal{L}^2} \leq C\int_{\mathbb{R}^d}\psi_\infty\left|\mathscr{F}(u\cdot\nabla g)\right|^2 dq + C\int_{\mathbb{R}^d}\psi_\infty\left|\mathscr{F}(\nabla u\cdot{q}g\psi_{\infty})\right|^2 dq.~~~~
    \end{align}
Applying Lemma \ref{Lemma1}, we have
\begin{align}\label{5ineq7}
	|\hat{\tau}|^2 &= \left(\int_{\mathbb{R}^d} q \otimes\nabla_q\mathcal{U} \hat{g} \psi_{\infty} dq\right)^2 \leq C\|\nabla_q\hat{g}\|^2_{\mathcal{L}^2}.
\end{align}
Adding \eqref{5ineq5} to $\lambda\times\eqref{5ineq6}$ with $\lambda$ large enough and integrating $\xi$ over $S_0(t)$, we deduce that
	\begin{align}\label{5ineq8}
		\int_{S_0(t)}|\hat{u}(t,\xi)|^2+\lambda\|\hat{g}\|^2_{\mathcal{L}^2}d\xi
		\leq \int_{S_0(t)} |\hat{u}_0|^2
		&+\lambda\|\hat{g}_0\|^2_{\mathcal{L}^2}d\xi+C\int_{S_0(t)}\int_{0}^{t}|\mathscr{F}(u\otimes u)|^2 ds'd\xi  \\ \notag
		+\lambda\int_{S_0(t)}\int_{0}^{t}\int_{\mathbb{R}^d}\psi_\infty|\mathscr{F}(u\cdot\nabla\hat{g})|^2 dqds'd\xi&+\lambda\int_{S_0(t)}\int_{0}^{t}\int_{\mathbb{R}^d}\psi_\infty|\mathscr{F}(\nabla u\cdot{q}\hat{g}\psi_{\infty})|^2 dqds'd\xi.
	\end{align}
	Under the assumptions in Theorem \ref{th3}, we get
	\begin{align}\label{5ineq9}
		~~~~~\int_{S_0(t)} |\hat{u}_0|^2
		+\|\hat{g}_0\|^2_{\mathcal{L}^2}d\xi
		&\leq\int_{S_0(t)} d\xi\cdot\||\hat{u}_0|^2
		+\|\hat{g}\|^2_{\mathcal{L}^2}\|_{L^{\infty}(S(t))} \\ \notag
		&\leq C\left(\frac {f'(t)} {f(t)}\right)^{\frac d 2}\left(\|u_0\|^2_{L^1}+\|\hat{g}_0\|^2_{L^1(\mathcal{L}^2)}\right).
	\end{align}
	According to Minkowski's inequality and \eqref{1ineq1}, we have
	\begin{align}\label{5ineq10}
		\int_{S_0(t)}\int_{0}^{t}|\mathscr{F}(u\otimes u)|^2 ds'd\xi
		&=\int_{0}^{t}\int_{S_0(t)}|\mathscr{F}(u\otimes u)|^2 d\xi ds'  \\ \notag
		&\leq C\int_{S_0(t)}d\xi \int_{0}^{t}\||\mathscr{F}(u\otimes u)|^2\|_{L^{\infty}}ds' ~~~\\ \notag
		&\leq C\ln^{-1}(e+t).
	\end{align}
	Using ${\rm div}~u=0$ and \eqref{1ineq1}, we obtain
		\begin{align}\label{5ineq11}
		\int_{S_0(t)}\int_{0}^{t}\int_{\mathbb{R}^d}\psi_\infty|\mathscr{F}(u\cdot\nabla g)|^2 dqds'd\xi
		&\leq C\int_{S_0(t)}|\xi|^2d\xi \int_{0}^{t}\int_{\mathbb{R}^d}\|\psi_\infty|\mathscr{F}(u g)|^2\|_{L^{\infty}}dqds' \\ \notag
		&\leq C\left(\frac {f'(t)} {f(t)}\right)^{\frac d 2+1} \int_{0}^{t}\|u\|^2_{L^{2}}\|g\|^2_{L^{2}(\mathcal{L}^{2})}ds' \\ \notag
		&\leq C\left(\frac {f'(t)} {f(t)}\right)^{\frac d 2+1}.
	\end{align}
Applying Proposition \ref{prop2} and Lemma \ref{Lemma1}, we infer that
	\begin{align}\label{5ineq12}
		~\int_{S_0(t)}\int_{0}^{t}\int_{\mathbb{R}^d}\psi_\infty|\mathscr{F}(\nabla u\cdot{q}g)|^2 dqds'd\xi
		&\leq C\int_{S_0(t)}d\xi \int_{0}^{t}\int_{B}\|\psi_\infty|\mathscr{F}(\nabla u\cdot{q}g)|^2\|_{L^{\infty}}dqds' ~~\\ \notag
		&\leq C\left(\frac {f'(t)} {f(t)}\right)^{\frac d 2} \int_{0}^{t}\|\nabla u\|^2_{L^{2}}\|\langle q\rangle g\|^2_{L^{2}(\mathcal{L}^{2})}ds' \\ \notag
		&\leq C\left(\frac {f'(t)} {f(t)}\right)^{\frac d 2}.
	\end{align}
	Combining the estimates for \eqref{5ineq8}, we conclude that
	\begin{align}\label{5ineq13}
		\int_{S_0(t)}|\hat{u}(t,\xi)|^2 d\xi\leq C\ln^{-1}(e+t).~~~~~
	\end{align}
	According to \eqref{5ineq3''} and \eqref{5ineq13}, we deduce that
	\begin{align}\label{5ineq14}
		\frac d {dt} \|u\|^2_{L^2}+C_d\frac {f'(t)} {f(t)}\| u\|^2_{L^2}
		\leq CC_d\frac {f'(t)} {f(t)} \ln^{-1}(e+t).~~~~~
	\end{align}
	By performing a routine procedure, one can arrive at
	\begin{align}\label{5ineq15}
		~~~~~~~~~~\|u\|^2_{L^2}\leq  C\ln^{-1}(e+t).
	\end{align}
	Using the initial decay \eqref{5ineq15}, we improve the time decay rate in $L^2$ by using bootstrap argument.
	\begin{align}\label{5ineq16}
		\int_{S(t)}\int_{0}^{t}|\mathscr{F}(u\otimes u)|^2 dsd\xi
		&\leq C\left(\frac {f'(t)} {f(t)}\right)^{\frac d 2} \int_{0}^{t}\|u\|^4_{L^{2}}ds' \\ \notag
		&\leq C\left(\frac {f'(t)} {f(t)}\right)^{\frac d 2} \int_{0}^{t}\ln^{-2}(e+s')ds' \\ \notag
		&\leq C\ln^{-3}(e+t).
	\end{align}
	Then the proof of \eqref{5ineq14} implies that
	\begin{align}\label{5ineq17}
		\|u\|^2_{L^2} \leq  C\ln^{-3}(e+t).
	\end{align}
\textbf{Step 2 :} Define $S(t)=\{\xi:|\xi|^2\leq C_d(1+t)^{-1}\}$, which will be useful to prove polynomial decay. By Schonbek's strategy and \eqref{5ineq2''}, we infer that
	\begin{align}\label{5ineq18}
		\frac d {dt} \|u\|^2_{L^2}+\frac {C_d} {1+t}\| u\|^2_{L^2}
		\leq \frac { C_d} {1+t}\int_{S(t)}|\hat{u}(\xi)|^2 d\xi + Ce^{-2at}.
	\end{align}
By performing a routine procedure as \eqref{5ineq8}, one can arrive at
	\begin{align}\label{5ineq19}
		\int_{S(t)}|\hat{u}(t,\xi)|^2+\lambda\|\hat{g}\|^2_{\mathcal{L}^2}d\xi
		\leq \int_{S(t)} |\hat{u}_0|^2
		&+\lambda\|\hat{g}_0\|^2_{\mathcal{L}^2}d\xi+C\int_{S(t)}\int_{0}^{t}|\mathscr{F}(u\otimes u)|^2 ds'd\xi  \\ \notag
		+\lambda\int_{S(t)}\int_{0}^{t}\int_{\mathbb{R}^d}\psi_\infty|\mathscr{F}(u\cdot\nabla g)|^2 dqds'd\xi&+\lambda\int_{S(t)}\int_{0}^{t}\int_{\mathbb{R}^d}\psi_\infty|\mathscr{F}(\nabla u\cdot{q}g)|^2 dqds'd\xi.
	\end{align}
	Under the additional assumption in Theorem \ref{th3}, we deduce that
	\begin{align}\label{5ineq20}
		~~\int_{S(t)} |\hat{u}_0|^2
		+\lambda\|\hat{g}_0\|^2_{\mathcal{L}^2}d\xi
		&\leq\int_{S(t)} d\xi\cdot\left\||\hat{u}_0|^2
		+\lambda\|\hat{g}_0\|^2_{\mathcal{L}^2}\right\|_{L^{\infty}(S(t))} \\ \notag
		&\leq C(1+t)^{-\frac d 2}\left(\|u_0\|^2_{L^1}+\lambda\|g_0\|^2_{L^1(\mathcal{L}^2)}\right),
	\end{align}
	and
	\begin{align}\label{5ineq21}
		\int_{S(t)}\int_{0}^{t}|\mathscr{F}(u\otimes u)|^2 dsd\xi
		\leq C(1+t)^{-\frac d 2} \int_{0}^{t}\|u\|^4_{L^{2}}ds'.
	\end{align}
	Using ${\rm div}~u=0$ and \eqref{1ineq1}, we have
	\begin{align}\label{5ineq22}
		\int_{S(t)}\int_{0}^{t}\int_{\mathbb{R}^d}\psi_\infty|\mathscr{F}(u\cdot\nabla g)|^2 dqds'd\xi
		&\leq C(1+t)^{-\frac d 2-1} \int_{0}^{t}\|u\|^2_{L^{2}}\|g\|^2_{L^{2}(\mathcal{L}^{2})}ds' \\ \notag
		&\leq C(1+t)^{-\frac d 2-1}.
	\end{align}
Applying Proposition \ref{prop2} and Lemma \ref{Lemma1}, we obtain
	\begin{align}\label{5ineq23}
		~\int_{S(t)}\int_{0}^{t}\int_{\mathbb{R}^d}\psi_\infty|\mathscr{F}(\nabla u\cdot{q}g)|^2 dqds'd\xi
		&\leq C(1+t)^{-\frac d 2} \int_{0}^{t}\|\nabla u\|^2_{L^{2}}\|\langle q\rangle g\|^2_{L^{2}(\mathcal{L}^{2})}ds'~~~~~~~~ \\ \notag
		&\leq C(1+t)^{-\frac d 2}.
	\end{align}
	Combining the estimates for \eqref{5ineq19}, we conclude that
	\begin{align}\label{5ineq24}
		\int_{S(t)}|\hat{u}(t,\xi)|^2 d\xi\leq C\left((1+t)^{-\frac d 2}+(1+t)^{-\frac d 2} \int_{0}^{t}\|u\|^4_{L^{2}}ds'\right).
	\end{align}
	According to \eqref{5ineq18} and \eqref{5ineq24}, we deduce that
	\begin{align}\label{5ineq25}
		\frac d {dt} \|u\|^2_{L^2} + \frac {C_d} {1+t}\| u\|^2_{L^2}
		\leq CC_d (1+t)^{-\frac d 2-1} \left(1+\int_{0}^{t}\|u\|^4_{L^{2}}ds'\right).
	\end{align}
 Multiplying $(1+t)^{\frac d 2+1}$ to \eqref{5ineq25} and integrating over $[0,t]$, one can arrive at
	\begin{align}\label{5ineq26}
		(1+t)^{\frac d 2+1}\|u\|^2_{L^2}&\leq Ct+C\int_{0}^{t}\int_{0}^{s'}\|u\|^4_{L^{2}}ds''ds'.~~~~
	\end{align}
	Define $M(t)=\mathop{\sup}\limits_{ s'\in[0,t]} (1+s')^{\frac d 2} \|u\|^2_{L^2}(s')$. Using \eqref{5ineq17} and \eqref{5ineq26}, we deduce that
	\begin{align}\label{5ineq27}
		M(t) \leq C+\int_{0}^{t}M(s')(1+s')^{-\frac d 2}\ln^{-3}(e+s')ds'.~~
	\end{align}
	Applying Gronwall's inequality, then we get $M(t)\leq C$ for any $t>0$,
	which implies that
	\begin{align}\label{5ineq28}
		\|u\|_{L^2}\leq  C(1+t)^{-\frac d 4}.~~~~~~~~~~~~
	\end{align}
\textbf{Step 3 :} We end up with establishing the lower bound of $L^2$ decay rate. Taking Leray projector $\mathbb{P}$ and Fourier transformation with respect to $x$ in system $(\ref{eq0})_1$, we infer that
	\begin{align}\label{5ineq29}
	~~\hat{u}_t + |\xi|^2\hat{u} = i\xi\widehat{\mathbb{P}\tau} - \widehat{\mathbb{P}(u\cdot\nabla u)}.
\end{align}
Integrating over $[0,t]$ with respect to $s$, one can arrive at
	\begin{align}\label{5ineq31}
	\hat{u}  =e^{-|\xi|^2t}\hat{u}_0 + \int_0^t e^{-|\xi|^2(t-s)}i\xi\big[\widehat{\mathbb{P}\tau} - \widehat{\mathbb{P}(u\otimes u)}\big]ds.
\end{align}
Under conditions of Theorem \ref{th1} that $\int_{\mathbb{R}^d}u_0dx \neq 0$, we can choose a ball $B$ containing the origin such that
$\displaystyle\inf_{\xi\in B}\hat{u}_0 \geq c_0$ for some positive constant $c_0$. Denote that $d_0=cc^2_0$ for some $c$ small enough. According to Minkowski inequality, we deduce that
	\begin{align}\label{5ineq32}
	\left(\int_{B}|\hat{u}|^2d\xi\right)^{\frac{1}{2}} &\geq \left(\int_{B}e^{-|\xi|^2t}|\hat{u}_0|^2d\xi\right)^{\frac{1}{2}} - \int_0^t \|e^{-|\xi|^2(t-s)}i\xi\big[\widehat{\mathbb{P}\tau} - \widehat{\mathbb{P}(u\otimes u)}\big]\|_{L^2(B)}ds \\ \notag
	& \geq \left(\int_{B}e^{-|\xi|^2t}|\hat{u}_0|^2d\xi\right)^{\frac{1}{2}} - \int_0^t \|e^{-|\xi|^2(t-s)}i\xi\|_{L^{\frac{2d}{d-2}}}\|\widehat{\mathbb{P}\tau} - \widehat{\mathbb{P}(u\otimes u)}\|_{L^d}ds \\ \notag
	& \geq d_0(1+t)^{-\frac{d}{4}} - \int_0^t (1+t-s)^{-\frac{d}{4}}\left(\|\tau\|_{L^{\frac{d}{d-1}}} +\|u\otimes u\|_{L^{\frac{d}{d-1}}}\right)ds.
\end{align}
Denote that
$$
  E^{\alpha,\beta} = \|\left(u_0,\|g_0\|_{\mathcal{L}^2}\right)\|^{\alpha}_{L^1\cap L^2}\|\left(u_0,\|g_0\|_{\mathcal{L}^2}\right)\|^{\beta}_{L^2},~~~~~~~~~~~~~
$$
and
$$
B_d = \int_0^t (1+t-s)^{-\frac{d}{4}}\left(\|\tau\|_{L^{\frac{d}{d-1}}} +\|u\otimes u\|_{L^{\frac{d}{d-1}}}\right)ds.
$$
Under conditions of Theorem \ref{th3}, we can take  $\|(u_0,\|g_0\|_{\mathcal{L}^2})\|_{L^2}$ small enough such that
$$
E^{0,1}+E^{\frac{1}{3},\frac{2}{3}}+E^{1,1} +E^{\frac{11}{6},\frac{1}{6}} \leq \frac{d_0}{2C}.~~~~~~
$$
For $d=3$, we infer from \eqref{5ineq28}, Lemma \ref{Ldecay} and Proposition \ref{prop1} that
	\begin{align}\label{5ineq33}
~~~~~~~~~~~~~	B_3
	&\leq \int_0^t (1+t-s)^{-\frac{3}{4}}\left(\|\tau\|^{\frac{1}{3}}_{L^{1}}\|\tau\|^{\frac{2}{3}}_{L^{2}} +\|u\|_{L^{2}}\|\nabla u\|_{L^{2}}\right)ds \\ \notag
	& \leq CE^{\frac{1}{3},\frac{2}{3}}(1+t)^{-\frac{3}{4}} + CE^{1,1}\big[\int_0^t (1+t-s)^{-\frac{3}{2}}(1+t)^{-\frac{3}{2}}ds\big]^{\frac{1}{2}} \\ \notag
	& \leq  \frac{d_0}{2}(1+t)^{-\frac{d}{4}}.
\end{align}
Consider the critical case $d=2$, we need more integrability in time for $\|\nabla u\|_{L^2}$. Let's recall the $L^2$ energy estimate as follows,
	\begin{align}\label{5ineq34}
	~\frac{d}{dt}\|u\|^2_{L^2} + \|\nabla u\|^2_{L^2} \leq C\|\tau\|^2_{L^2}.
\end{align}
Multiplying $(1+t)^{\frac{1}{2}}$ to \eqref{5ineq34}, we obtain
	\begin{align}\label{5ineq35}
	\frac{d}{dt}(1+t)^{\frac{1}{2}}\|u\|^2_{L^2} + (1+t)^{\frac{1}{2}}\|\nabla u\|^2_{L^2} \leq  C(1+t)^{\frac{1}{2}}\|\tau\|^2_{L^2} + \frac{1}{2}(1+t)^{-\frac{1}{2}}\| u\|^2_{L^2}
\end{align}
Integrating over $[0,t)$ with respect to $s$, we infer from \eqref{5ineq2''} and \eqref{5ineq28} that
	\begin{align}\label{5ineq36}
	~~~(1+t)^{\frac{1}{2}}\|u\|^2_{L^2} + \int_0^t (1+s)^{\frac{1}{2}}\|\nabla u\|^2_{L^2}ds \leq CE^{2,0}.
\end{align}
For $d=2$, applying Lemma \ref{Ldecay}, \eqref{5ineq28} and \eqref{5ineq36}, we deduce that
\begin{align}\label{5ineq37}
	B_2
	&\leq \int_0^t (1+t-s)^{-\frac{1}{2}}\left(\|\tau\|_{L^{2}} +\|u\|_{L^{2}}\|\nabla u\|_{L^{2}}\right)ds \\ \notag
	& \leq  CE^{0,1}\int_0^t (1+t-s)^{-\frac{1}{2}}e^{-ct} ds + CE^{\frac{11}{6},\frac{1}{6}}\big[\int_0^t (1+t-s)^{-1}(1+t)^{-\frac{4}{3}}ds\big]^{\frac{1}{2}}\\ \notag
	& \leq  \frac{d_0}{2}(1+t)^{-\frac{1}{2}}.
\end{align}
According to \eqref{5ineq32}, \eqref{5ineq33} and \eqref{5ineq37}, we conclude that
	\begin{align}\label{5ineq38}
	\|u\|_{L^2} \geq \left(\int_{B}|\hat{u}|^2d\xi\right)^{\frac{1}{2}} &\geq \frac{d_0}{2}(1+t)^{-\frac{d}{4}},~~~~~~~~
\end{align}
which implies that the decay rate we obtain is optimal.
\hfill$\Box$ \\
\smallskip
\noindent\textbf{Acknowledgments} This work was
partially supported by the National Natural Science Foundation of China (No.12171493), the Macao Science and Technology Development Fund (No. 0091/2018/A3), and Guangdong Province of China Special Support Program (No. 8-2015).

%The authors thank the referee for valuable comments and suggestions.

\phantomsection
\addcontentsline{toc}{section}{\refname}
%添加参考文献到书签，宏包 hyperref
\bibliographystyle{abbrv} %plain ,%alpha, %abbrv
\bibliography{hooke}

\end{document}